\documentclass[11pt]{amsart}
\usepackage{amsmath, amssymb, amscd, mathrsfs, slashed, enumerate, url}
\usepackage{color}
\usepackage{xcolor}
\usepackage[all,cmtip]{xy}
\usepackage{multicol}
\usepackage{indentfirst}
\usepackage{latexsym}
\usepackage{bm}
\usepackage{graphicx}
\usepackage{subfigure,esint}
\usepackage{float}
\usepackage{verbatim}
\usepackage{comment}
\usepackage[top=1in, bottom=1in, left=1.25in, right=1.25in]{geometry}
\usepackage{epsfig,dsfont,amsthm,amsfonts,amsbsy}
\usepackage[colorlinks]{hyperref}
\usepackage[toc,page]{appendix}
\usepackage{tikz-cd}

\usepackage{enumitem}

\newcommand{\gpp}{\mathfrak{g}_P}

\newcommand{\gph}{\mathfrak{g}_{P_h}}

\newcommand{\MGC}{\mathcal{G}_{\mathbb{C}}}

\newcommand{\lan}{\langle }
\newcommand{\ran}{\rangle}

\newcommand{\ML}{\mathcal{L}}

\newcommand{\da}{\dagger}

\newcommand{\MM}{\mathcal{M}}

\newcommand{\MO}{\mathcal{O}}

\newcommand{\Ker}{\mathrm{Ker}}
\newcommand{\Vol}{\mathrm{Vol}}

\newcommand{\End}{\mathrm{End}}

\newtheorem{theorem}{Theorem}[section]

\newtheorem{conjecture}[theorem]{Conjecture}
\newtheorem{corollary}[theorem]{Corollary}

\newtheorem{definition}[theorem]{Definition}
\newtheorem{question}[theorem]{Question}
\newtheorem{example}[theorem]{Example}

\newtheorem{lemma}[theorem]{Lemma}

\newtheorem{proposition}[theorem]{Proposition}
\newtheorem*{remark}{Remark}

\renewcommand{\Im}{\mathrm{Im}}

\newcommand{\MC}{\mathcal{C}}
\newcommand{\Tr}{\mathrm{Tr}}

\newcommand{\st}{\star}
\newcommand{\we}{\wedge}
\newcommand{\pa}{\partial}

\newcommand{\Ric}{\mathrm{Ric}}

\newcommand{\vp}{\varphi}
\newcommand{\MA}{\mathcal{A}}
\newcommand{\ME}{\mathcal{E}}

\newcommand{\na}{\nabla}
\newcommand{\ep}{\epsilon}
\newcommand{\hA}{\widehat{A}}

\newcommand{\hP}{\widehat{\phi}}
\newcommand{\bpa}{\bar{\pa}}

\newcommand{\Hom}{\mathrm{Hom}}
\newcommand{\be}{\beta}

\newcommand{\mfg}{\mathfrak{g}}

\newcommand{\lam}{\lambda}

\newcommand{\mfe}{\mathfrak{e}}

\newcommand{\diag}{\mathrm{diag}}

\newcommand{\MG}{\mathcal{G}}

\newcommand{\su}{\mathfrak{su}}

\newcommand{\MF}{\mathcal{F}}

\newcommand{\Lam}{\Lambda}

\newcommand{\MMH}{\MM_{\mathrm{Higgs}}}
\newcommand{\MMHS}{\MM_{\mathrm{HS}}}

\newcommand{\hphi}{\hat{\phi}}
\newcommand{\Zan}{Z_{\mathrm{T}}}
\newcommand{\loc}{\mathrm{loc}}

\newcommand{\MB}{\mathcal{B}}
\newcommand{\MI}{\mathcal{I}}
\newcommand{\ZT}{\mathbb{Z}/2}

\newcommand{\app}{\mathrm{app}}
\newcommand{\Sym}{\mathrm{Sym}}

\newcommand{\MRE}{\mathrm{E}}

\newcommand{\rank}{\mathrm{rank}}
\newcommand{\Id}{\mathrm{Id}}

\newcommand{\Uh}{\mathrm{Uh}}

\newcommand{\Ta}{\mathrm{T}}

\newcommand{\ch}{\mathrm{ch}}
\newcommand{\Zta}{Z_{\mathrm{T}}}

\newcommand{\Aut}{\mathrm{Aut}}

\newcommand{\hvp}{\hat{\vp}}
\newcommand{\dist}{\mathrm{dist}}

\newcommand{\vol}{\mathrm{vol}}

\newcommand{\MMP}{\MM_{\mathrm{PS}}}

\newcommand{\hbe}{\hat{\be}}
\newcommand{\Op}{\mathrm{Op}}
\newcommand{\mbe}{\mathbf{e}}

\newcommand{\Jac}{\mathrm{Jac}}

\newcommand{\mbs}{\mathbf{s}}
\newcommand{\bv}{\mathbf{v}}

\usepackage[utf8]{inputenc}
\usepackage[english]{babel}
\usepackage{fancyhdr}
\begin{document}
	
	\title[The behavior of sequences of solutions to the Hitchin-Simpson equations]{The behavior of sequences of solutions to the Hitchin-Simpson equations}
	
	\author[He]{Siqi He}
	\email{sqhe@amss.ac.cn}
	\address{Morningside Center of Mathematics,
		Chinese Academy of Sciences, 
		Beijing, 100190 China}
	
	\maketitle
\begin{abstract}
	The Hitchin-Simpson equations are first-order non-linear equations for a pair consisting of a connection and a Higgs field. In this paper, we study the behavior of sequences of solutions to the Hitchin-Simpson equations on closed K\"ahler manifolds with unbounded $L^2$ norms of the Higgs fields. We prove a compactness result for the connections and renormalized Higgs fields, which generalizes the work of Taubes \cite{taubes2013compactness} and Mochizuki \cite{Mochizukiasymptotic}.
	
	As applications, we prove that every $\mathbb{Z}/2$ harmonic 1-form on a K\"ahler manifold can be deformed into a sequence of solutions to the Hitchin-Simpson equations. Additionally, we solve the generalized Hitchin's WKB problem on any K\"ahler manifold.
\end{abstract}

	\begin{section}{Introduction}
The study of the compactness problem for flat $\mathrm{SL}_2(\mathbb{C})$ connections on 3- and 4-manifolds was pioneered by Taubes \cite{Taubes20133manifoldcompactness,taubes2013compactness}. In four dimensions, Taubes also proved a compactness theorem for $\mathrm{SL}_2(\mathbb{C})$ connections in both the Kapustin-Witten equations \cite{taubes2013compactness} and the Vafa-Witten equations \cite{taubes2017behavior}. We also refer to \cite{haydyswalpuski2015compactness,walpuskizhang2019compactness} for related compactness results. Additionally, on Riemann surfaces, flat $\mathrm{SL}_2(\mathbb{C})$ connections are closely related to the theory of Higgs bundles and Hitchin fibrations. Through the theory of Higgs bundles and tools from complex geometry, a series of papers have provided significant insights into the ends of Taubes' compactified space, cf. \cite{mazzeo2012limiting,mazzeo2016ends,fredrickson2018generic,Mochizukiasymptotic,hecompactification2023}.

This paper attempts to generalize Taubes' compactness results to higher-dimensional K\"ahler manifolds and to $\mathrm{GL}_r(\mathbb{C})$ connections that satisfy the Hitchin-Simpson equations. The Hitchin-Simpson equations were introduced by Hitchin \cite{hitchin1987self} on Riemann surfaces and studied by Simpson \cite{Simpson1988Construction,Simpson1992} on general K\"ahler manifolds. These equations interact with various areas of mathematics, including gauge theory, the theory of Higgs bundles, and representations of the fundamental group. Our approach combines key contributions from Taubes \cite{Taubes20133manifoldcompactness,taubes2013compactness} on $\mathrm{SL}_2(\mathbb{C})$ connections and Mochizuki \cite{Mochizukiasymptotic} on the asymptotic behavior of Higgs bundles on Riemann surfaces. Furthermore, there have been recent developments regarding the compactness of the Vafa-Witten equations by Chen \cite{chen2024vafawitten} and Taubes \cite{taubes2024vafawitten}.

Let $(X, \omega)$ be a closed K\"ahler manifold of real dimension $2n$ with K\"ahler class $\omega$, and let $E$ be a rank $r$ complex vector bundle with Hermitian metric $H$. Let $A$ be a unitary connection, and $\phi$ be a Hermitian 1-form valued in $\End(E)$. We write $d_A = \bar{\pa}_A + \pa_A$ and $\phi = \vp + \vp^{\da}$, where $\vp \in \Omega_X^{1,0}(\End(E))$. The Hitchin-Simpson equations for $A$ and $\phi$ can then be written as
\begin{equation*}
	\begin{split}
		F_A^{2,0} = 0,\;\bar{\pa}_A \vp = 0,\; \vp \wedge \vp = 0,\\
		\sqrt{-1} \Lambda (F_A^{1,1} + [\vp, \vp^{\da}]) = \gamma \Id_E,
	\end{split}
\end{equation*}
where $\Lambda: \Omega^{1,1}(\End(E)) \to \Omega^0(\End(E))$ is the contraction with the K\"ahler class, and $\gamma$ is a constant defined as $\gamma = \frac{2\pi \deg(E)}{(n-1)!r}$. Additionally, we can define $D = d_A + \phi$. When $\deg(E) = 0$ and $\ch_2(E) \cdot [\omega]^{n-2} = 0$, the Hitchin-Simpson equations are equivalent to the flat connection equation $D^2 = 0$.

Recall that a Higgs bundle on a K\"ahler manifold consist of $(E,\bar{\pa}_E, \vp)$, where $\bar{\pa}_E^2=0$, which defines a holomorphic structure on $E$ and $\vp \in H^0(\Omega_X^{1,0}(\End(E)))$ satisfies the integrability condition $\vp \wedge \vp = 0$.

Compared to the study of compactness problems on 3- and 4-manifolds by Taubes \cite{Taubes20133manifoldcompactness}, studying the compactness problem on K\"ahler manifolds offers two significant advantages. First, we have the Kobayashi-Hitchin correspondence \cite{donaldson1985anti,hitchin1987self,Simpson1988Construction}, which establishes a one-to-one correspondence between the moduli space of polystable Higgs bundles and the moduli space of solutions to the Hitchin-Simpson equations.

The second advantage is the Hitchin morphism, introduced by Hitchin \cite{hitchin1987stable} and further studied in \cite{simpson1994moduli,chenngo2020hitchin}, which plays an important role in our work. Let $\MMH$ be the moduli space of polystable Higgs bundles, and let $p_1, \cdots, p_r$ be a basis of invariant polynomials of $\End(E)$. For any $[(\bar{\pa}_E, \vp)] \in \MMH$, we define the Hitchin base $\MA_X := \oplus_{i=1}^r H^0(\Sym^i \Omega_X^{1,0})$, and the Hitchin morphism is given by
$$
\kappa: \MMH \to \MA_X, \quad \kappa(\vp) := (p_1(\vp), \cdots, p_r(\vp)).
$$
The image of the Hitchin base, called the spectral data, reflects the eigenvalues of the Higgs field. The discriminant locus is a complex analytic subvariety that corresponds roughly to points in $X$ where some eigenvalues of $\vp$ have higher multiplicities. Outside the discriminant locus, the Higgs bundle decomposes the vector bundle problem into a line bundle problem, which greatly simplifies the analysis.
		
\subsection{Main Result}
Our main result is a compactness theorem for the Hitchin-Simpson equations, which can be stated as follows:

\begin{theorem}
	\label{Thm_maintheorem}
	Let $X$ be a closed K\"ahler manifold, and let $E$ be a rank $r$ vector bundle with a Hermitian metric $H$. Let $(A_i, \phi_i)$ be a sequence of solutions to the Hitchin-Simpson equations \eqref{Hitchin-Simpson}, and let $r_i := \|\phi_i\|_{L^2(X)}$.
	
	\begin{itemize}
		\item [(i)] If the sequence $\{r_i\}$ has a bounded subsequence, then there exist a Hausdorff codimension-4 singular set $Z_{\Uh}$, a rank $r$ Hermitian vector bundle $(E_{\infty}, H_{\infty})$, a unitary connection $A_{\infty}$, and a Hermitian 1-form $\phi_{\infty}$, such that $(A_{\infty}, \phi_{\infty})$ satisfies the Hitchin-Simpson equations on $X \setminus Z_{\Uh}$. Moreover, passing to a subsequence and up to unitary gauge, $(A_i, \phi_i)$ converges to $(A_{\infty}, \phi_{\infty})$ in the $\mathcal{C}^{\infty}_{\text{loc}}$ sense over any compact subset of $X \setminus Z_{\Uh}$.
		
		\item [(ii)] If the sequence $\{r_i\}$ has no bounded subsequence, define $\hat{\phi}_i := r_i^{-1} \phi_i = \hat{\varphi}_i + \hat{\varphi}_i^{\dagger}$ and let $\mbs_i := \kappa(\hat{\varphi}_i) \in \MA_X$ be the image under the Hitchin morphism. Then $\mbs_i$ converges smoothly on $X$ to $\mbs_{\infty}$ in $\MA_X$.
		
		Let $Z_{\Ta} \subset X$ be the discriminant locus defined by $\mbs_{\infty}$. Suppose $Z_{\Ta} \subsetneq X$ (which is automatically satisfied when $\rank\; E = 2$), then:
		\begin{itemize}
			\item [(a)] There exists a Hausdorff codimension-4 subset $Z_{\Uh} \subset X \setminus Z_{\Ta}$, a unitary connection $A_{\infty}$, and a Hermitian 1-form $\phi_{\infty}$ such that, on $X \setminus (Z_{\Uh} \cup Z_{\Ta})$, up to unitary gauge and passing to a subsequence, $(A_i, \hat{\phi}_i)$ converges to $(A_{\infty}, \phi_{\infty})$ in the $\mathcal{C}^{\infty}_{\text{loc}}$ topology.
			\item [(b)] $A_{\infty}$ is a Hermitian-Yang-Mills connection, and $(A_{\infty}, \phi_{\infty})$ satisfies the adiabatic Hitchin-Simpson equations \eqref{adabatic-Hitchin-Simpson}.
			\item [(c)] If $\deg(E) = 0$ and $\ch_2(E) \cdot [\omega]^{n-2} = 0$, then $Z_{\Uh} = \emptyset$, and for each $U \subset X \setminus Z_T$, there exists a constant $d_U$ such that, passing to a subsequence, for each $k$, there exist constants $C, C'$, depending on $k$ and $U$, such that over $U$, we have exponential convergence
			\begin{equation}
				\|(A_i, \hat{\phi}_i) - (A_{\infty}, \phi_{\infty})\|_{\mathcal{C}^k} \leq C e^{-C' r_i d_U}.
			\end{equation}
		\end{itemize}
	\end{itemize}
\end{theorem}

The discriminant locus roughly consists of points in $X$ corresponding to the common eigenvalues of the Higgs fields, which will be explicitly defined in Chapter \ref{sec_Hitchin_morphism_the_spectral_variety}. When $\rank\;E \geq 3$ and $\Tr(\phi_i) \neq 0$, it is possible that $Z_{\Ta} = X$, and in this situation, Theorem \ref{Thm_maintheorem} does not apply. Furthermore, part (i) of the above theorem is well-known and follows from classical elliptic theory and Uhlenbeck compactness. The $\mathcal{C}^{\infty}$ convergence for similar equations on 3- and 4-manifolds has also been obtained by Parker \cite{parker2023concentrating}.
		
\subsection{Applications}
We introduce two applications of the compactness theorem proved above: the deformation problem of $\ZT$ harmonic 1-forms and the generalized Hitchin's WKB problem.

\subsubsection{Deformation problem of $\ZT$ harmonic 1-forms}
In \cite{Taubes20133manifoldcompactness,taubes2013compactness}, Taubes introduced the concept of a $\ZT$ harmonic 1-form, which is roughly the boundary of the compactification of a flat $\mathrm{SL}_2(\mathbb{C})$ connection. These $\ZT$ harmonic 1-forms are closely related to various topics in differential geometry, gauge theory, 3-manifold topology, hyperK\"ahler geometry, and special Lagrangian submanifolds \cite{doan2017existence,donaldson21deformation,takahashi2015moduli,zhang2017rectifiability,he2024Z2harmonic,he23branched}.

A $\ZT$ harmonic 1-form is a two-valued 1-form $\bv = \pm v$ defined on $X \setminus Z$, where $Z$ is a codimension-2 subset. The form is harmonic on $X \setminus Z$ and satisfies $\int_{X \setminus Z} |\nabla \bv|^2 < \infty$. On a K\"ahler manifold, $\ZT$ harmonic forms are relatively easier to understand. Despite the $\pm$ ambiguity, $s_{\bv} := v^{1,0} \otimes v^{1,0}$ is a well-defined holomorphic section of $H^0(\Sym^2 \Omega_X^{1,0})$. Furthermore, under this identification, all $\ZT$ harmonic 1-forms can be identified with rank-one symmetric differentials, as introduced by Bogomolov-Oliveria \cite{bogomolov2011symmetric}, and are closely related to the Hitchin morphism and Chen-Ngo's description \cite{chenngo2020hitchin}. We refer to Section \ref{sec_Hitchin_morphism_the_spectral_variety} for more detailed explanations. When $\rank\;E = 2$ and $\Tr(\phi_i) = 0$, the adiabatic Hitchin-Simpson equations yield a $\ZT$ harmonic 1-form.

A $\ZT$ harmonic 1-form $\bv$ is said to be deformable if there exists a sequence $(A_i, \phi_i)$ with unbounded $L^2$ norms, i.e., $\lim_{i \to \infty} \|\phi_i\|_{L^2(X)} = \infty$, such that over $X \setminus Z$, for some Hausdorff codimension-2 singular set $Z$, the limit $\lim_{i \to \infty} (A_i, \phi_i)$ exists and $\lim_{i \to \infty} \frac{1}{\|\phi_i\|_{L^2(X)}^2}\Tr((\phi_i^{1,0})^2) = s_{\bv}$.

The deformation problem is a fundamental question in the analytic study of gauge-theoretic equations. We are particularly interested in the following deformation question, which has been posed for various compactifications of gauge-theoretic equations and has been solved for Hitchin equations on Riemann surfaces \cite{mazzeo2012limiting} and two-spinor Seiberg-Witten equations \cite{doan2020deformation,parker2024gluing,parker2023deformations}. On closed 3-manifolds, it has been shown in \cite{he2024Z2harmonic} that not every $\ZT$ harmonic 1-form is deformable.

\begin{question}(Deformation Problem.)
	Which $\ZT$ harmonic 1-forms on $X$ can be deformed into a sequence of solutions to the Hitchin-Simpson equations?
\end{question}

On a closed K\"ahler manifold, using the Hitchin section construction for rank-two Higgs bundles \cite{heliu2023spectralvariety}, we prove the following:

\begin{theorem}
	\label{thm_deformation_ZT_harmonic_form1}
	Every $\ZT$ harmonic 1-form on a K\"ahler manifold can be deformed into a sequence of solutions to the Hitchin-Simpson equations that converge to $\bv$.
\end{theorem}

We are particularly interested in the deformation problem on 4-manifolds. Taubes proved a compactness result for the Kapustin-Witten equations \cite{taubes2013compactness}, while Tanaka \cite{tanaka2019singular} showed that on a K\"ahler surface, the Kapustin-Witten equations are equivalent to the Hitchin-Simpson equations. In \cite{bogomolov2011symmetric}, Bogomolov-Oliveira constructed a $\ZT$ harmonic 1-form on a simply connected closed K\"ahler surface $X_{\mathrm{BO}}$. Applying Theorem \ref{thm_deformation_ZT_harmonic_form1} to the $\ZT$ harmonic 1-form on $X_{\mathrm{BO}}$ yields the following interesting corollary:

\begin{corollary}
	On the 4-manifold $X_{\mathrm{BO}}$, there exists a $\ZT$ harmonic 1-form that is the limit of a sequence of solutions to the Kapustin-Witten equations but not the limit of a sequence of solutions to flat $\mathrm{SL}_2(\mathbb{C})$ connections.
\end{corollary}

\subsubsection{Generalized Hitchin's WKB problem}
When $X$ is a Riemann surface, the Hitchin's WKB problem was first introduced by Katzarkov-Noll-Pandit-Simpson \cite{katzarkov2015harmonic} and was solved by Mochizuki \cite{Mochizukiasymptotic} for Riemann surfaces. We also refer \cite{segman2024localasymptotics,sagman2022unstable,biswas2021holomorphic} for generalization and applications for the WKB problem on Riemann surface.

Let $(E, \bar{\pa}_E, t\vp)$ be a stable Higgs bundle with $\deg(E) = 0$ and $\ch_2(E) \cdot [\omega]^{n-2} = 0$, where $t$ is a real parameter. Let $(A_t, \phi_t)$ be the corresponding flat solutions to the Hitchin-Simpson equations under the Kobayashi-Hitchin correspondence. We define $D_t = d_{A_t} + \phi_t$, so that $D_t$ is a family of flat connections. Moreover, we assume that the discriminant locus $Z$ of the Higgs bundle $(E, \bar{\pa}_E, \vp)$ is a proper subset, i.e., $Z \subsetneq X$.

For the family of connections $D_t$, the Hitchin's WKB problem seeks to understand the asymptotic behavior of the monodromy defined by $D_t$. More specifically, let $\gamma: [0,1] \to X \setminus Z$ be a smooth embedded path, and let $P_t$ denote the parallel transport between $\gamma(0)$ and $\gamma(1)$. Let $H_{0,t} (H_{1,t})$ be the harmonic metric at $\gamma(0)$ $(\gamma(1))$, and let $P_t^{\st} H_{1,t}$ denote the pullback metric under the parallel transport. 

For the Hermitian metrics $H_{0,t}$ and $P_t^{\st} H_{1,t}$, we can take a basis $e_1, \cdots, e_r$, orthogonal with respect to both metrics. Then we define numbers $\mu_i := \log |e_i|_{P_t^{\st} H_{1,t}} - \log |e_i|_{H_{0,t}}$. Assume that $\mu_1 \geq \cdots \geq \mu_r$. The distance vector is defined as
$$
\vec{d}(H_0, P_t^{\st} H_1) := (\mu_1, \cdots, \mu_r).
$$

We aim to understand the following question. Since the original question for the Riemann surface is called Hitchin's WKB problem \cite[P. 11]{katzarkov2015harmonic}, we refer to the following asymptotic problem as the generalized Hitchin's WKB problem.

\begin{question}(Generalized Hitchin's WKB problem)
	\label{question_WKB}
	What is the asymptotic behavior of $P_t$ as $t \to \infty$? More precisely, what is the asymptotic behavior of the distance vector $\vec{d}(H_0, P_t^{\st} H_1)$?
\end{question}

Let $U_{\gamma} \subset X \setminus Z$ be a neighborhood of $\gamma$. Then there exists a decomposition of the Higgs bundle $(E, \bar{\pa}_E, \vp) = \oplus_{i=1}^r (E_i, \bpa_i, \vp_i)$, where $\rank\;E_i = 1$ and $\vp_i$ are holomorphic 1-forms. Moreover, for $i \neq j$, $\vp_i - \vp_j$ has no zeros. Let $s$ be the coordinate on $[0,1]$, and for the path $\gamma: [0,1]_s \to X \setminus Z$, we write $\gamma^* \vp_i = a_i ds$. The path $\gamma$ is called non-critical if $\Re(a_i(s)) \neq \Re(a_j(s))$. Assume $\Re(a_1(s)) < \cdots < \Re(a_r(s))$ for $s \in [0,1]$. We define $\alpha_i := -\int_0^1 \Re(a_i(s)) ds$. Then, we have the following asymptotic result, which solves the generalized Hitchin's WKB problem:

\begin{theorem}
	Suppose $\gamma$ is non-critical. Then there exist positive constants $C, C', t_0$, depending on $(\bar{\pa}_E, \vp)$, such that for any non-critical path $\gamma$, we have
	\begin{equation}
		\left| \frac{1}{t} \vec{d}(H_0, P_t^{\st} H_1) - (2 \alpha_1, \cdots, 2 \alpha_r) \right| \leq C e^{-C' t},
	\end{equation}
	for $t \geq t_0$.
\end{theorem}

\textbf{Conventions.} In all estimates below, $C$, $C'$, etc., are constants that depend only on $X$, $E$, and the background Hermitian metric $H$ on $E$, as well as the Sobolev constants. However, their values may change from one line to the next. The notation $L^p_k$ refers to $L^p$ functions with weak derivatives up to order $k$ have a finite $L^p$ norm. Whenever a constant depends on additional data, we will specify this explicitly. Throughout this paper, $X$ always denotes a closed K\"ahler manifold with $\Vol(X) = 1$ and $\dim_{\mathbb{R}}(X) = 2n$. We typically denote $R_0$ as the injectivity radius of $X$.

\textbf{Acknowledgements.} The author wishes to thank Mark Cataldo, Tsao-Hsien Chen, Xuemiao Chen, Alexander Doan, Simon Donaldson, Andriy Haydys, Rafe Mazzeo, Takuro Mochizuki, Nathaniel Sagman, Clifford Taubes, Richard, Wentworth, Song Sun and Boyu Zhang for numerous helpful discussions. Part of this work was completed while the author was visiting Stanford University and IGP, and the author is grateful to the Stanford math department and IGP for their hospitality. S.H. is partially supported by NSFC grant No. 12288201 and No. 2023YFA1010500.

	\end{section}

\begin{section}{Higgs Bundles and the Hitchin-Simpson Equations on K\"ahler Manifolds}
	In this section, we introduce the Kobayashi-Hitchin correspondence for Higgs bundles and the Hitchin-Simpson equations on closed K\"ahler manifolds, following Hitchin \cite{hitchin1987self} and Simpson \cite{Simpson1988Construction}.
	
	\begin{subsection}{Higgs Bundles on K\"ahler Manifolds}
		Let $E$ be a complex vector bundle over a K\"ahler manifold $X$. We denote by $\Omega_X^{p,q}(E)$ the space of $E$-valued $(p,q)$-forms on $X$. The holomorphic structures on $E$ can be identified with connections $\bar{\pa}_E:\Omega^0_X(E) \to \Omega^{0,1}_X(E)$ that satisfy the integrability condition $\bar{\pa}_E^2 = 0$. We usually write $\ME = (E, \bar{\pa}_E)$ to denote a holomorphic vector bundle with holomorphic structure defined by $\bar{\pa}_E$.
		
		\begin{definition}
			A Higgs bundle on $X$ is a pair $(\bar{\pa}_E, \vp)$ such that:
			\begin{itemize}
				\item [(i)] $\bar{\pa}_E: \Omega_X^0(E) \to \Omega_X^{0,1}(E)$ is a connection with $\bar{\pa}_E^2 = 0$, defining a holomorphic structure on $E$.
				\item [(ii)] $\vp: E \to E \otimes \Omega_X^{1,0}$ is holomorphic, i.e., $\bar{\pa}_E \vp = 0$. We write $\vp \in H^0(\Omega_X^{1,0}(\End(E)))$ and call $\vp$ the Higgs field.
				\item [(iii)] $\vp \wedge \vp = 0$.
			\end{itemize}
		\end{definition}
		
		For notation simplification, we usually write $(\bar{\pa}_E,\vp)$ for a Higgs bundle instead of $(E,\bar{\pa}_E,\vp)$. We now introduce the concept of stability for Higgs bundles.
		\begin{definition}
			For any holomorphic bundle $\ME = (E, \bar{\pa}_E)$, we define the slope of the bundle as $\mu(\ME) := \frac{\deg \ME}{\rank E}$. A Higgs bundle $(\bar{\pa}_E, \vp)$ is called \emph{stable} if, for every $\vp$-invariant coherent subsheaf $\MF \subset \ME$, where $\vp(\MF) \subset \MF \otimes \Omega_X^{1,0}$, we have $\mu(\MF) < \mu(\ME)$. A Higgs bundle $(\bar{\pa}_E, \vp)$ is called \emph{polystable} if it is a direct sum of stable Higgs bundles with the same slope.
		\end{definition}
		
		We define the complex gauge group $\MGC := \Aut(E)$, which acts on Higgs bundles as follows: for $g \in \MGC$, $g(\bar{\pa}_E, \vp) := (g^{-1} \circ \bar{\pa}_E \circ g, g^{-1} \circ \vp \circ g)$. This action preserves the stability condition. The orbits of polystable Higgs bundles under $\MGC$ are closed. The moduli space of polystable Higgs bundles is defined as
		$$
		\MMH := \{\text{Polystable Higgs bundles } (\bar{\pa}_E, \vp)\}/\MGC.
		$$
		
		We denote by $[(\bar{\pa}_E, \vp)]$ the equivalence class under the $\MGC$ orbit in $\MMH$.
	\end{subsection}
		
	\subsection{The Hitchin-Simpson Equations and the Kobayashi-Hitchin Correspondence}
	Now, we introduce the Hitchin-Simpson equations, which were defined by Hitchin \cite{hitchin1987self} for Riemann surfaces and generalized by Simpson for K\"ahler manifolds \cite{Simpson1988Construction}.
	
	\subsubsection{Hermitian Geometry on K\"ahler Manifolds}
	Let $X$ be a closed K\"ahler manifold with K\"ahler metric $g$. We denote the corresponding K\"ahler class by $\omega$. In an orthonormal frame $dz_i$ on $X$ with $|dz_i|^2 = 2$, we write $\omega = \frac{\sqrt{-1}}{2} \sum_{i=1}^n dz_i \wedge d\bar{z}_i$. Then, we can define the contraction map with the K\"ahler class, $\Lambda: \Omega_X^{p,q} \to \Omega_X^{p-1,q-1}$, which acts on $\Omega_X^{1,1}$ by 
	$$
	\Lambda \left( \sum_{\alpha\beta} a_{\alpha\beta} dz_{\alpha} \wedge d\bar{z}_{\beta} \right) = -2\sqrt{-1} \sum_{i=1}^n a_{ii}.
	$$ 
	Additionally, we have $\Lambda \alpha \cdot \vol_g = \frac{1}{(n-1)!} \alpha \wedge \omega^{n-1}$ for $\alpha \in \Omega_X^{1,1}$.
	
	Let $E$ be a rank $r$ complex vector bundle with Hermitian metric $H$, and let $d_A$ be a connection on $E$. The connection $d_A$ is called unitary with respect to the Hermitian metric $H$ if, for any sections $s_1, s_2$ of $E$, we have 
	$$
	dH(s_1, s_2) = H(d_A s_1, s_2) + H(s_1, d_A s_2).
	$$ 
	If we decompose $d_A = \partial_A + \bar{\partial}_A$, the connection $A$ is unitary if and only if 
	$$
	dH(s_1, s_2) = H(\bar{\partial}_A s_1, s_2) + H(s_1, \partial_A s_2).
	$$
	
	Given a Higgs bundle $(\bar{\partial}_E, \varphi)$, we can define the adjoints $\partial_H$ and $\varphi^{\dagger}$ as follows:
	\begin{equation*}
		\partial H(s_1, s_2) = H(\bar{\partial}_E s_1, s_2) + H(s_1, \partial_H s_2), \quad H(\varphi s_1, s_2) = H(s_1, \varphi^{\dagger} s_2),
	\end{equation*}
	where $s_1, s_2$ are sections of $E$. Given any Hermitian metric, we can define a unitary connection $d_A := \bar{\partial}_E + \partial_H$ and a Hermitian 1-form $\phi := \varphi + \varphi^{\dagger}$. We denote $D := d_A + \phi$; then $D$ is a complex connection with curvature $D^2 = F_A^{1,1} + [\varphi, \varphi^{\dagger}]$. If $D^2 = 0$, we call $D$ a \emph{flat connection}.
	
	For any connection $\bar{\partial}_A: \Omega_X^{p,q}(E) \to \Omega_X^{p,q+1}(E)$ and $\partial_A: \Omega_X^{p,q}(E) \to \Omega_X^{p+1,q}(E)$, using the metric on $E$ and the K\"ahler metric on $X$, we can define the adjoint operators $\bar{\partial}_A^{\ast}$ and $\partial_A^{\ast}$, which satisfy the K\"ahler identities:
	$$
	\partial_A^{\ast} = i[\Lambda, \bar{\partial}_A], \quad \bar{\partial}_A^{\ast} = i[\Lambda, \partial_A].
	$$
	
	Moreover, for the Laplacian $\Delta$ acting on $\Omega_X^0$, we have $\Delta = d^{\ast} d = 2\partial^{\ast} \partial = 2\bar{\partial}^{\ast} \bar{\partial} = 2\sqrt{-1}\Lambda \bar{\partial} \partial$. Additionally, for every $s \in \Omega_X^{1,0}(E)$, we denote by $s^{\dagger}$ the adjoint of $s$, and we have $\sqrt{-1}\Lambda (s \wedge s^{\dagger}) = |s|^2$.
	
	\subsubsection{The Hitchin-Simpson Equations on K\"ahler Manifolds}
	Let $(E, H)$ be a rank $r$ Hermitian vector bundle. The Hitchin-Simpson equations are equations for a unitary connection $A$ and a Hermitian 1-form $\phi = \varphi + \varphi^{\dagger}$:
	\begin{equation}
		\begin{split}
			&F_A^{0,2} = 0, \quad \varphi \wedge \varphi = 0, \quad \bar{\partial}_A \varphi = 0, \\
			&i \Lambda (F_A^{\perp} + [\varphi, \varphi^{\dagger}]) = 0,
		\end{split}
		\label{Hitchin-Simpson}
	\end{equation}
	where $F_A^{\perp} = F_A^{1,1} - \gamma(E) \Id$ is the trace-free part of $F_A^{1,1}$, and $\gamma(E) = \frac{2\pi \deg(\ME)}{(n-1)!r}$.
	
	The first three equations in \eqref{Hitchin-Simpson} are invariant under the complex gauge group $g \in \MGC$, which defines a Higgs bundle $(\bar{\partial}_A, \varphi)$. Additionally, $(d_A, \phi)$ is the Chern connection for $(\bar{\partial}_A, \varphi)$ together with the Hermitian metric $H$. Therefore, we can also regard \eqref{Hitchin-Simpson} as equations for a Higgs bundle and a Hermitian metric $(\bar{\partial}_E, \varphi, H)$.
	
	The complex gauge group $g \in \MGC$ acts on $(A, \phi)$ as follows:
	\begin{equation*}
		\begin{split}
			d_{A^g} &= \bar{\partial}_{A^g} + \partial_{A^g} = g^{-1} \circ \bar{\partial}_A \circ g + g^{\dagger} \circ \partial_A \circ (g^{\dagger})^{-1}, \\
			\phi^g &= \varphi^g + (\varphi^g)^{\dagger} = g^{-1} \circ \varphi \circ g + g^{\dagger} \circ \varphi^{\dagger} \circ (g^{\dagger})^{-1}.
		\end{split}
	\end{equation*}
	
	The complex gauge group $\MGC$ also acts on the Hermitian metric as $H^g(s_1, s_2) = H(gs_1, gs_2)$ for any sections $s_1$ and $s_2$. We now state the following proposition that explains the relationship between these two descriptions:
	\begin{proposition}
		\label{Proposition_idenpent}
		Let $E$ be a complex vector bundle over $X$. Suppose $(\bar{\partial}_A, \varphi, H)$ is a solution to the Hitchin-Simpson equations. Then $(\bar{\partial}_{A^g}, \varphi^g, H^g)$ is also a solution to the Hitchin-Simpson equations.
		
		Let $A^g, \phi^g$ be the Chern connection of $(\bar{\partial}_{A^g}, \varphi^g, H^g)$. Then we have $|F_{A^g}|_{H^g} = |F_A|_H$ and $|\varphi^g \wedge \varphi^g|_{H^g} = |\varphi \wedge \varphi|_H$.
	\end{proposition}
	\begin{proof}
		Let $s_1, s_2$ be two sections of $E$. We compute
		\begin{equation*}
			\begin{split}
				H^g(\bar{\partial}^g s_1, s_2) &= H(g \bar{\partial}^g s_1, g s_2) = H(\bar{\partial}(g s_1), g s_2) \\
				&= \partial H(g s_1, g s_2) - H(g s_1, \partial_A(g s_2)) = \partial H^g(s_1, s_2) - H^g(s_1, g^{-1} \partial_A g s_2).
			\end{split}
		\end{equation*}
		Therefore, using the Hermitian metric $H^g$ for the adjoint, we have $d_{A^g} = g^{-1} \circ d_A \circ g$. Similarly, we have $\phi^g = g^{-1} \circ \phi \circ g$. Writing $\phi^g = \varphi^g + (\varphi^g)^{\dagger}$, we obtain $F_{A^g} = g^{-1} \circ F_A \circ g$ and $[\varphi^g, (\varphi^g)^{\dagger}] = g^{-1} \circ [\varphi, \varphi^{\dagger}] \circ g$. The proposition follows immediately.
	\end{proof}
	
	We define the unitary gauge group $\MG := \{g \in \MGC \mid g g^{\dagger} = \Id_E\}$, which preserves \eqref{Hitchin-Simpson}. The moduli space of the Hitchin-Simpson equations $\MMHS$ is defined as
	\begin{equation}
		\MMHS := \{(A, \phi) \mid (A, \phi) \text{ satisfy } \eqref{Hitchin-Simpson}\}/\MG.
	\end{equation}
	
	Additionally, there is a natural map $\Xi: \MMHS \to \MMH$ that forgets the second equation of \eqref{Hitchin-Simpson}. We define $(A, \phi)$ as irreducible if $\Ker d_A|_{\Omega_X^0} = \Ker [\phi, \cdot]|_{\Omega_X^0} = 0$, and completely reducible if $(A, \phi)$ are direct sums of irreducible pairs.
	
	\begin{theorem}{\cite{hitchin1987self}, \cite[Theorem 1]{Simpson1988Construction}}
		\label{thm_nonabelian_Hogde}
		The map $\Xi$ is one-to-one. Explicitly, given any Higgs bundle $(\bar{\partial}_E, \varphi)$, there exists an irreducible solution $(A, \phi)$ to the Hitchin-Simpson equations \eqref{Hitchin-Simpson} with $[(\bar{\partial}_A, \phi^{1,0})] = [(\bar{\partial}_E, \varphi)] \in \MMH$ if and only if the Higgs bundle is stable, and a completely reducible solution if and only if it is polystable.
	\end{theorem}
	
	\subsubsection{Topologically Trivial Higgs Bundles}
	For every pair $(A, \phi = \varphi + \varphi^{\dagger})$ with $\bar{\partial}_A \varphi = 0$ and $\varphi \wedge \varphi = 0$, we define the energy function of $(A, \phi)$ as
	$$
	\MRE(A, \phi) = \int_X |F_A + [\varphi, \varphi^{\dagger}]|^2 + |d_A \phi|^2,
	$$
	then we have the following proposition:
	
	\begin{proposition}{\cite[Proposition 3.4]{Simpson1988Construction}}
		\label{energyidentity}
		Let $(A, \phi)$ be a solution to the Hitchin-Simpson equations \eqref{Hitchin-Simpson}. Then
		$$
		\MRE(A, \phi) = 4\pi^2 \int_X \ch_2(E) \wedge \frac{\omega^{n-2}}{(n-2)!} + \gamma^2 \rank(E),
		$$
		where $\gamma:=\frac{2\pi \deg(E)}{(n-1)!r}$.
	\end{proposition}
	\begin{proof}
		We have the following energy identity:
		\begin{equation}
			\label{Eq_Energyidentities}
			\MRE(A, \phi) = \int_X |\sqrt{-1}\Lambda(F_A + [\varphi, \varphi^{\dagger}] - \gamma \Id)|^2 + \gamma^2 \rank(E) + 4\pi^2 \int_X \ch_2(E) \wedge \frac{\omega^{n-2}}{(n-2)!}.
		\end{equation}
	\end{proof}
	
	\begin{corollary}{\cite[Proposition 3.4]{Simpson1988Construction}}
		\label{Cor_Whenaconnectionflat}
		Let $(A, \phi)$ be a solution to the Hitchin-Simpson equations \eqref{Hitchin-Simpson}. Suppose $\deg(E) = 0$ and $\ch_2(E) \cdot [\omega]^{n-2} = 0$. Then \eqref{Hitchin-Simpson} is equivalent to the connection $d_A + \varphi + \varphi^{\dagger}$ being flat:
		\begin{equation}
			\label{eq_flat_connection}
			\begin{split}
				&F_A^{0,2} = 0, \quad \varphi \wedge \varphi = 0, \quad \bar{\partial}_A \varphi = 0, \\
				&F_A^{1,1} + [\varphi, \varphi^{\dagger}] = 0.
			\end{split}
		\end{equation}
	\end{corollary}
	
	\begin{definition}
		A Higgs bundle $(\bar{\partial}_E, \varphi)$ is called \emph{topologically trivial} if $\deg(E) = 0$ and $\ch_2(E) \cdot [\omega]^{n-2} = 0$.
	\end{definition}
	
	\end{section}
	
\begin{section}{Hitchin Morphism and the Spectral Variety}
	\label{sec_Hitchin_morphism_the_spectral_variety}
	In this section, we introduce the spectral geometry of the Higgs bundle moduli space, which was first introduced by Hitchin on Riemann surfaces \cite{hitchin1987self,hitchin1987stable} and later studied by Chen-Ngo \cite{chenngo2020hitchin} for higher-dimensional varieties. For further insights into the Hitchin morphism, we refer to \cite{heliu2023spectralvariety,heliumok2023rigidity}.
	
	\subsection{Hitchin Morphism}
	Let $(\bar{\pa}_E, \varphi)$ be a Higgs bundle over $X$. Since $\varphi \in H^0(X, \Omega^{1,0}_X(\End(\ME)))$, if $p_k$ is an invariant polynomial of $\End(\ME)$ with $\deg(p_k) = k$, it defines a map 
	$$
	p_k: H^0(X, \Omega^{1,0}_X(\End(\ME))) \to H^0(\Sym^k \Omega^{1,0}_X),
	$$ 
	where $\Sym^k \Omega_X^{1,0}$ denotes the symmetric product of $\Omega_X^{1,0}$. We define the Hitchin base 
	$$
	\MA_X := \oplus_{i=1}^r H^0(X, \Sym^i \Omega_X^{1,0}),
	$$ 
	and the Hitchin morphism as:
	\begin{equation}
		\begin{split}
			\kappa &: \MMH \to \MA_X; \\
			\kappa([(\bar{\pa}_E, \varphi)]) &= (p_1(\varphi), \cdots, p_r(\varphi)).
		\end{split}
	\end{equation}
	
	When $\dim(X) = 1$, the condition $\varphi \wedge \varphi = 0$ is automatically satisfied, and the Hitchin morphism becomes a fibration \cite{hitchin1992lie}, forming an algebraic integrable system. In higher dimensions, the additional condition $\varphi \wedge \varphi = 0$ imposes strong restrictions on the image of the Hitchin morphism, so $\kappa$ is not surjective.
	
	We now briefly recall the definition of the spectral base, introduced by Chen-Ngo \cite{chenngo2020hitchin} and studied in \cite{heliu2023spectralvariety,heliumok2023rigidity}.
	
	\begin{definition}
		\label{d.spectralbase}
		The spectral base $\MB_X$ is the subset of $\MA_X$ consisting of elements $\textbf{s} = (s_1, \dots, s_r) \in \MA_X$ such that for any point $x\in X$, there exist $r$ elements $\lambda_1, \dots, \lambda_r \in \Omega_X^{1,0}|_x$ satisfying $s_i(x) = \sigma_i(\lambda_1, \dots, \lambda_r)$, where $\sigma_i$ is the $i$-th elementary symmetric polynomial in $r$ variables. An element $\textbf{s} \in \MB_X$ is called a spectral datum.
	\end{definition}
	
	Roughly, given $\mbs = (s_1, \cdots, s_r) \in \MB_X$, we may think of it as corresponding to an $r$-valued holomorphic $(1,0)$-form $(\lambda_1, \cdots, \lambda_r)$ such that $s_i = \sigma(\lambda_1, \dots, \lambda_r)$. In \cite{chenngo2020hitchin}, Chen-Ngo proved that $\MB_X$ is a closed subscheme. Moreover, the Hitchin morphism factors through the spectral base as follows:
	
	\begin{proposition}[\protect{\cite[Proposition 5.1]{chenngo2020hitchin}}]
		\label{prop_spectralbase}
		The Hitchin morphism $\kappa: \MMH \to \MA_X$ factors through the natural inclusion map $\iota_X: \MB_X \to \MA_X$, such that the following diagram commutes:
		\begin{equation}
			\label{e.spectralmorphism}
			\begin{tikzcd}[column sep=large,row sep=large]
				\MMH \arrow[d,"{}" left] \arrow{dr}{\kappa} &  \\
				\MB_X \arrow[r,"{\iota_X}" below] & \MA_X.
			\end{tikzcd}
		\end{equation}
	\end{proposition}
	\begin{proof}
		Let $(\bar{\pa}_E, \varphi)$ be a Higgs bundle with $\Tr(\varphi^i) = s_i \in H^0(X, \Sym^i \Omega_X^1)$. At any point $x\in X$, let $dz^1, dz^2, \cdots, dz^n$ be a local frame of $\Omega^{1}_X$ at $x\in X$. After choosing a frame for $E$, we can write $\varphi(x) = \sum_{i=1}^n B_i dz^i$, where the $B_i$ are $r \times r$ matrices. The condition $\varphi \wedge \varphi = 0$ implies $[B_i, B_j] = 0$ for all $1 \leq i \leq j$. Thus, the $B_i$ can be simultaneously triangularized, and so can $\varphi(x)$ as an 1-form valued $r \times r$ matrix. In particular, after changing local coordinates, we may assume that $\varphi(x)$ is an upper triangular matrix, and let $\lambda_1, \dots, \lambda_r \in \Omega^{1,0}_{X}|_x$ be its diagonal elements. By the definition of the Hitchin morphism, we have $s_i(x) = \sigma_i(\lambda_1, \dots, \lambda_r)$, which proves the claim.
	\end{proof}
	
	\subsection{The Spectral Variety}
	We now define the spectral variety. Given $\mbs = (s_1, \cdots, s_r) \in \MB_X$, the spectral variety $S_{\mbs}$ of $X$ is defined as
	\begin{equation}
		S_{\mbs} := \{\lambda \in \Omega_X^{1,0} \mid \lambda^r + p_1 \lambda^{r-1} + \cdots + p_r = 0\}.
	\end{equation}
	Locally, we can choose a trivialization of $\Omega_X^{1,0} \cong \mathbb{C}^n$ and consider the defining equation of $S_{\mbs}$ on each component of $\mathbb{C}^n$. The integrability condition $\varphi \wedge \varphi = 0$ ensures that the definition of the spectral variety is not overdetermined and independent of the choice of trivialization. For further discussion, see \cite{Katzarkov2013notes,chenngo2020hitchin}.
	
	Let $(\bar{\pa}_E, \varphi)$ be a Higgs bundle. For each $x\in X$, we define $\lambda \in \Omega_X^{1,0}$ as an eigenvalue of $\varphi$ if there exists a section $s \in E|_x$ such that $\varphi s = \lambda s$. Moreover, as $\bar{\pa}_E\vp=0$, $\lambda$ will also be holomorphic. The condition $\varphi \wedge \varphi = 0$ guarantees that at every point of $X$, $\varphi$ has well-defined eigenvalues. Therefore, given a Higgs bundle $(\bar{\pa}_E, \varphi)$ with spectral cover $S_{\mbs}$, the points in $S_{\mbs}|_x$ correspond to the $r$ eigenvalues of $\varphi|_x$.
	
	Let $\lambda_1, \lambda_2, \cdots, \lambda_r$ be the $r$ eigenvalues of $\varphi$ and write $\mbs := \kappa(\varphi) \in \MB_X$. We define the discriminant section as:
	\begin{equation}
		\begin{split}
			\Delta_{\mbs}: X \to \Sym^{r^2 - r} \Omega_X^{1,0}, \quad p \mapsto \sum_{1 \leq i < j \leq r} (\lambda_i(x) - \lambda_j(x))^2.
			\label{Eq_discriminantsection}
		\end{split}
	\end{equation}
	The zero set of the discriminant section is called the \emph{discriminant locus}, denoted as $Z_{\mbs}$ or $Z_{\vp}$. The spectral variety is ramified along the discriminant locus for the projection $\pi: S_{\mbs} \to X$. Since $Z_{\mbs}$ is the zero section of the holomorphic section $\Delta_{\mbs}$, it is a complex analytic subvariety.
	
	There are many examples of Higgs bundles $(\bar{\pa}_E, \varphi)$ where the discriminant locus $Z_{\vp} = X$, particularly for $\rank\;E = 3$. For any rank-2 Higgs bundle $(\bar{\pa}_E, \varphi)$, let $(L, 0)$ be a rank-1 Higgs bundle with a vanishing Higgs field. Then, the discriminant locus of $(E \oplus L, \varphi \oplus 0)$ will be $X$. For discussions on compactness results in this situation, see \cite{segman2024localasymptotics}. 
	
	We would like to consider the following class of Higgs bundles, which generically $\vp$ has $r$ different eigenvalues:
	\begin{definition}
		A Higgs bundle $(\bar{\pa}_E,\vp)$ is called \emph{generically semi-simple} if the discriminant locus $Z_{\vp}\neq X$.
	\end{definition}
	
	\subsection{Rank-Two Spectral Base}
	We now consider the spectral base for rank-2 traceless Higgs bundles. We define the moduli space of rank-2 traceless Higgs bundles as:
	\begin{equation}
		\label{eq_moduli_ranktwo_traceless}
		\begin{split}
			\MMH^2 := \{(\bar{\pa}_E, \varphi) \in \MMH \mid \mathrm{rank}\; E = 2, \; \Tr(\varphi) = 0\}/\MGC.
		\end{split}
	\end{equation}
	We will now study the image of the Hitchin morphism on the moduli space of rank-2 traceless Higgs bundles $\MMH^2$. We first consider:
	\begin{equation}
		\label{eq_spectral_base_Ranktwo}
		\MB_X^2 = \{s \in H^0(\Sym^2 \Omega_X^{1,0}) \mid \rank(s) \leq 1\},
	\end{equation}
	where for $x\in X$, $s|_x$ is a second symmetric tensor, and $\rank(s) \leq 1$ means that either $s|_x = 0$ or $s|_x = v \otimes v$ for $v \in \Omega_X^{1,0}|_x$.
	
	\begin{proposition}
		\label{Prop_discriminantlocus}
		For the restriction of the Hitchin morphism, we have $\kappa(\MMH^2) \subset \MB_X^2$. Moreover, for $\mbs \neq 0 \in \MB_X^2$, the discriminant locus $Z_{\mbs}$ is a complex analytic subvariety with codimension at least 2.
	\end{proposition}
	\begin{proof}
		For $(\bar{\pa}_E, \varphi) \in \MMH^2$, $\varphi$ is a rank-2 traceless Higgs bundle. The eigenvalues of $\varphi$ come in pairs $-\lambda, \lambda$, and $\kappa(\varphi) = \Tr(\varphi^2) = 2(\lambda \otimes \lambda) \in \MB_X^2$.
		
		For $\mbs \neq 0 \in \MB_X^2$, $Z_{\mbs}$ is the zero set of $\mbs \in H^0(\Sym^2 \Omega_X^1)$, which is a complex analytic subvariety with codimension at least 2.
	\end{proof}
	
	Thus, since $\kappa(\MMH^2) \subset \MB_X^2$, we call $\MB_X^2$ the \emph{rank-two spectral base}, which serves as the spectral base for rank-2 traceless Higgs bundles. The spectral base $\MB_X^2$ has interesting geometric meaning, as it represents the space of rank-1 symmetric differentials and has been studied by Bogomolov-Oliveria \cite{bogomolov2011symmetric}. When $X = \Sigma$ is a Riemann surface, let $K_{\Sigma}$ be the canonical bundle. Then, $\MB_{\Sigma}^2 = H^0(K_{\Sigma}^2)$ is the space of quadratic differentials. Thus, the spectral base $\MB_X^2$ can be understood as a generalization of quadratic differentials to higher-dimensional varieties.
	
	\subsection{$\ZT$ Harmonic 1-Forms}
	\label{subsection_ZTharmonic_1form}
	The theory of $\ZT$ harmonic 1-forms was developed by Taubes \cite{Taubes20133manifoldcompactness} and has been studied from various perspectives \cite{takahashi2015moduli, zhang2017rectifiability}. We now discuss the relationship between $\ZT$ harmonic 1-forms and the rank-two spectral base $\MB_X^2$ and refer \cite[Section 4]{heparker2024notes} for similar discussions.
	
	We first recall the definition of $\ZT$ harmonic 1-forms:
	\begin{definition}
		\label{def_ZTharmonicform}
		A $\mathbb{Z}/2$ harmonic 1-form on $X$ is given by a closed subset $Z \subsetneq X$ and a two-valued section $\bv$ of $T^*X$ on the complement of $Z$, such that the following conditions hold:
		\begin{enumerate}
			\item For each $p \notin Z$, there exists an open neighborhood $U_p \subset X \setminus Z$ such that the values of $\bv$ on $U$ are of the form $\pm v$, where $v$ is a non-vanishing section of $T^*X|_U$ with $dv = 0$ and $d^*v = 0$.
			\item For every open subset $U \subset X$ such that $\overline{U}$ is compact, we have $\int_{U \setminus Z} |\bv|^2 < +\infty$ and $\int_{U \setminus Z} |\nabla \bv|^2 < +\infty$.
		\end{enumerate}
	\end{definition}
	
	The concept of $\ZT$ harmonic 1-forms and the rank-two Hitchin base are equivalent. Given a $\ZT$ harmonic 1-form $\bv$, over each $U_p \subset M \setminus Z$, we write $\bv = \pm v$, and we define $s := v^{1,0} \otimes v^{1,0}$. Then $s$ is well-defined over $M \setminus Z$. As $v$ is harmonic, a straightforward computation using the K\"ahler identity shows that $\bar{\pa} s = 0$. The condition $\rank(s) \leq 1$ is automatically satisfied, so $s \in \MB_X^2$.
	
	The inverse construction is also straightforward. Given $s \in \MB_X^2$, let $Z_s := s^{-1}(0)$ be the discriminant locus. Then, over $X \setminus Z_s$, since $\rank(s) = 1$, a $\ZT$ harmonic 1-form $\bv$ is defined as $\bv := \pm \Re(\sqrt{s})$ on $X \setminus Z_s$.
	
	\subsection{Examples of Rank-Two Spectral Base}
	In this subsection, we present several examples of rank-two spectral bases, or equivalently, $\ZT$ harmonic 1-forms for specific K\"ahler manifolds.
	
	\begin{example}
		When $X = \Sigma$ is a Riemann surface with genus $g(\Sigma) \geq 2$, we have $\MB_X^2 = H^0(K_X^2)$, which is the space of quadratic differentials.
	\end{example}
	
	\begin{example}
		When $X$ is a rational surface, $K3$ surface, Enriques surface, or a hypersurface in $\mathbb{CP}^3$, it has been shown in \cite{sakai1979symmetric} that $H^0(\Sym^2 \Omega_X^{1,0}) = 0$, so $\MB_X^2 = 0$. Thus, there are no $\ZT$ harmonic 1-forms in these cases.
	\end{example}
	
	\begin{example}
		When $X = \Sigma_1 \times \Sigma_2$ with $\Sigma_1, \Sigma_2$ being Riemann surfaces of genus $g_1, g_2 \geq 1$, let $\pi_i: X \to \Sigma_i$ be the projection, and write $K_i := \pi_i^* K_{\Sigma_i}$. Then,
		$$
		\Omega_X^{1,0} = K_1 \oplus K_2, \quad H^0(X, \Sym^2 \Omega_X^{1,0}) = \oplus_{i,j=1}^2 H^{0}(X, K_i \otimes K_j).
		$$
		Given $s \in \MB_X^2$, we write $s = (s_{ij}) \in \oplus_{i,j=1}^2 H^0(X, K_i \otimes K_j)$. The condition $\rank(s) \leq 1$ is equivalent to $s_{11}s_{22} = s_{12}^2$. The discriminant locus of $s$ is $s^{-1}(0) = s_{11}^{-1}(0) \cap s_{22}^{-1}(0)$.
		
		If $s_{11} = 0$ but $s_{22} \neq 0$, then $s_{12} = 0$ and $Z = s^{-1}(0) = s_{22}^{-1}(0) = \Sigma_1 \times \{q_1, q_2, \dots, q_{4g_2-4}\}$, which has codimension 2. The corresponding $\ZT$ harmonic 1-form will be the pullback of a $\ZT$ harmonic 1-form from the Riemann surface. Similar behavior occurs when $s_{22} \neq 0$ and $s_{11} = 0$.
		
		If $s_{11}$ and $s_{22}$ are not identically zero, since $H^0(X, K_i^{\otimes 2}) = \pi_i^* H^0(\Sigma_i, K_{\Sigma_i}^{\otimes 2})$, we have $s_{11}^{-1}(0) = \{p_1, p_2, \dots, p_{4g_1-4}\} \times \Sigma_2$, and $s_{22}^{-1}(0) = \Sigma_1 \times \{q_1, q_2, \dots, q_{4g_2-4}\}$. Additionally, the condition $s_{12}^2 = s_{11}s_{22}$ requires that the multiplicities of $s_{11}^{-1}(0)$ and $s_{22}^{-1}(0)$ must be even. In this case, the singular set $Z$ for the corresponding $\ZT$ harmonic 1-form is $Z = 2\{p_1, \dots, p_{2g_1-2}\} \times 2\{q_1, \dots, q_{2g_2-2}\}$, which has codimension 4, and the corresponding line bundle over $X \setminus Z$ will extend globally to $X$.
	\end{example}
	
	\begin{example}
		Let $X$ be a ruled surface or a non-isotrivial elliptic surface with reduced fibers and a fibration $\pi: X \to \Sigma$, where $\Sigma$ has genus $g$. By \cite{chenngo2020hitchin}, the pullback map
		$$
		\pi^*: H^0(\Sigma, K_\Sigma^{\otimes 2}) \to H^0(X, \Sym^2 \Omega_X^{1,0})
		$$
		is an isomorphism, which implies $\MB_X^2 = \pi^* H^0(\Sigma,K_{\Sigma}^2)$. Thus, for any $s \in \MB_X^2$, we can write $s = \pi^* q$ with $q \in H^0(K_{\Sigma}^{\otimes 2})$, and $Z = s^{-1}(0) = \cup_{i=1}^{4g-4} \pi^{-1}(p_i)$, where $p_1, \dots, p_{4g-4}$ are the zeros of $q$. The corresponding $\ZT$ harmonic 1-form will be the pullback from the Riemann surface.
	\end{example}
	
	We now introduce an interesting example of the spectral base on a projective variety, due to Bogomolov and Oliveira.
	
	\begin{example}{\cite[Section 3]{bogomolov2011symmetric}}
		\label{example_bogomolov}
		In \cite{bogomolov2011symmetric}, an example of a compact, simply-connected, smooth projective surface $X$ is constructed such that $\MB_X^2 \neq 0$. We briefly recall the construction of $X$ from \cite{bogomolov2011symmetric}. Let $\mathbb{A}^3$ be a 3-dimensional abelian variety, and let $\hat{X}$ be a smooth hypersurface in $\mathbb{A}^3$. Let $\sigma = -\Id$ be the natural involution on $\mathbb{A}^3$, and assume that the action of $\sigma$ on $\hat{X}$ has only one fixed point $p$. Then, $\hat{X}/\sigma$ has an $A_2$-type singularity, and we define $X$ to be the minimal resolution of $\hat{X}/\sigma$.
		
		It is proved in \cite[Section 3]{bogomolov2011symmetric} that $\pi_1(X)$ is trivial. Moreover, there exists a unique holomorphic 1-form $\omega$ on $\mathbb{A}^3$ such that $\omega|_{T_p(\hat{X})} = 0$ (up to constant). Then, $\omega^2$ is a holomorphic section of $\Sym^2 \Omega_{\hat{X}/\sigma}^{1,0}$ on $\hat{X}/\sigma$, and it has a holomorphic extension $s$ to $X$. Additionally, $s \in \MB_X$, and it is a rank-two symmetric differential.
	\end{example}
	
	\subsection{Hitchin Section for Rank-Two Higgs Bundles}
	In this subsection, we introduce the construction of the Hitchin section for rank-2 traceless Higgs bundles over general projective varieties, developed by the author and Liu \cite{heliu2023spectralvariety}. From Proposition \ref{Prop_discriminantlocus}, we know that $\kappa(\MMH^2) \subset \MB_X^2$. We are now interested in whether the map $\kappa: \MMH^2 \to \MB_X^2$ is surjective.
	
	This key observation is explained by Bogomolov and Oliveira in \cite{bogomolov2011symmetric}:
	
	\begin{proposition}[\protect{\cite[p.1089]{bogomolov2011symmetric}}]
		\label{p.decomposition-symmetric-differetials}
		For $s \in \MB_X^2$, there exists a line bundle $\ML$ and holomorphic sections 
		$$
		\alpha \in H^0(\Omega_X^{1,0} \otimes \ML^{-1}), \quad \tau \in H^0(\ML^2)
		$$
		such that $s = \alpha^2 \tau$. Furthermore, the divisor defined by the zeros of $\tau$ is the same as the divisor defined by the zeros of $s$.
	\end{proposition}
	
	Thus, given $s \in \MB_X^2$, with $\alpha$ and $\tau$ as above, we can view $\alpha: \ML \to \Omega_X^{1,0}$ as a holomorphic map. We define $\beta := \frac{1}{2} \alpha \tau \in H^0(\Omega_X^{1,0} \otimes \ML)$, which can be viewed as a holomorphic map $\beta: \ML^{-1} \to \Omega_X^1$. Then, we define a Higgs bundle as follows:
	\begin{equation}
		\left(\ME_s = \ML \oplus \MO_X, \varphi_s = \begin{pmatrix}
			0 & \beta \\
			\alpha & 0
		\end{pmatrix}\right).
	\end{equation}
	We have $\kappa(\varphi_s) = \Tr(\varphi_s^2) = s$ and $\varphi_s \wedge \varphi_s = 0$. Moreover, by \cite[Proposition 5.7]{heliu2023spectralvariety}, for $s \neq 0 \in \MB_X^2$, $(\ME_s, \varphi_s)$ is polystable. Therefore, this defines a section of the map $\kappa: \MMH^2 \to \MB_X^2 \setminus \{0\}$. In summary, we have the following:
	
	\begin{theorem}{\cite[Theorem 1.2, Proposition 5.7]{heliu2023spectralvariety}}
		\label{thm_heliu_surjective_Hitchin_morphism}
		Given $s \in \MB_X$, there exists a traceless, rank-2, polystable Higgs bundle $(\ME_s, \varphi_s) \in \MMH^2$ such that $\kappa[(\ME_s, \varphi_s)] = s$.
	\end{theorem}
	
\end{section}

\begin{section}{$L^2_1$ Convergence and Eigenvalue Estimates}
	In this section, we initiate the analysis of sequences of solutions $(A_i, \phi_i)$ to \eqref{Hitchin-Simpson}. We will establish an $L^2_1$ estimate for $|\phi_i|$ and provide a $\MC^{\infty}$ estimate for the invariant polynomials associated with $\phi_i$.
		
	\begin{subsection}{Useful Identities}
		We begin by deriving some useful identities for solutions to the Hitchin-Simpson equations. Let $(E,H)$ be a Hermitian vector bundle over $X$, and let $(A, \phi = \vp + \vp^{\dagger})$ be a solution to the Hitchin-Simpson equations. Then, we have the following:
		
		\begin{lemma}
			\label{lemmalocalestimate}
			\label{Lemma_conjugatevanishing}
			\begin{itemize}
				\item [(i)] $\langle [\sqrt{-1}\Lambda[\vp, \vp^{\dagger}], \vp], \vp \rangle = \frac{1}{2} |[\vp, \vp^{\dagger}]|^2$.
				\item [(ii)] Let $s$ be a holomorphic section of $E$ with $[\vp, s] = 0$, then
				\begin{equation}
					\Delta |s|^2 + 2|\partial_A s|^2 + 2|[\vp^{\dagger}, s]|^2 = 0,
				\end{equation}
				\item [(iii)] Let $(A, \phi = \vp + \vp^{\dagger})$ be a solution to the Hitchin-Simpson equations \eqref{Hitchin-Simpson}. Denote by $\nabla_A$ the connection induced by $d_A$ and the Levi-Civita connection on $\Omega^1(\End(E))$. Then, the following identity holds:
				\begin{equation}
					\Delta |\phi|^2 + 2|\nabla_A \phi|^2 + 2|[\vp, \vp^{\dagger}]|^2 + 2\Ric(\phi, \phi) = 0,
					\label{Eq_integralbyparts}
				\end{equation}
				where, if we choose a real orthonormal frame $dx_i$ and write $\phi = \sum_{i=1}^{2n} \phi_i dx_i$, then $\Ric(\phi, \phi) = \sum_{i,j=1}^{2n} \Ric_{ij} \phi_i \phi_j$, and $\Ric_{ij}$ is the Ricci curvature on the plane spanned by $dx_i \wedge dx_j$.
			\end{itemize}
		\end{lemma}
		
		\begin{proof}
			For (i), let $A, B$ be any two matrices with $[A, B] = 0$. Then, $\Tr([[A, A^{\dagger}], B] B^{\dagger}) = \Tr([A^{\dagger}, B] [A, B^{\dagger}]) = \langle [A, B^{\dagger}], [A, B^{\dagger}] \rangle$. Locally, let $dz_1, dz_2, \dots, dz_n$ be an orthonormal basis of $\Omega_X^{1,0}$, and write $\vp = \sum_{i=1}^n \vp_i dz_i$. The condition $\vp \wedge \vp = 0$ implies $[\vp_i, \vp_j] = 0$. We then compute:
			\[
			\langle [\sqrt{-1}\Lambda[\vp, \vp^{\dagger}], \vp], \vp \rangle = \sum_{i,j=1}^n 2\langle [[\vp_i, \vp_i^{\dagger}], \vp_k], \vp_k \rangle.
			\]
			Since $[\vp_i, \vp_k] = 0$, it follows that $\langle [[\vp_i, \vp_i^{\dagger}], \vp_k], \vp_k \rangle = \langle [\vp_i, \vp_k^{\dagger}], [\vp_i, \vp_k^{\dagger}] \rangle$. The claim then follows as $|[\vp, \vp^{\dagger}]|^2 = 4\sum_{i,j} |[\vp_i, \vp_j^{\dagger}]|^2$.
			
			For (ii), let $s$ be a holomorphic section of the vector bundle $E$, with connection $d_A$ and curvature $\hat{F}_A$. We compute:
			\begin{equation}
				\begin{split}
					\Delta |s|^2 &= 2\bar{\partial}^{\dagger}\bar{\partial} |s|^2 = -2\sqrt{-1}\Lambda \partial \bar{\partial} |s|^2 = -2i \Lambda \partial \langle s, \partial_A s \rangle \\
					&= -2\sqrt{-1}\Lambda \langle \partial_A s, \partial_A s \rangle - 2\sqrt{-1}\Lambda \langle s, \bar{\partial}_A \partial_A s \rangle = -2|\partial_A s|^2 + 2i \Lambda \langle {F}_A^{1,1} s, s \rangle.
				\end{split}
			\end{equation}
			
			By \eqref{Hitchin-Simpson}, we have $\sqrt{-1}\Lambda F_A = -\sqrt{-1}\Lambda [\vp, \vp^{\dagger}] + \gamma(E)\Id_E$, where $\gamma(E) = \frac{2\pi\deg(E)}{(n-1)!r}$. Additionally, since $[\vp, s] = 0$, we compute:
			\begin{equation*}
				\begin{split}
					\langle 2i \Lambda F_A s, s \rangle &= -4\langle s, \sum_{i=1}^n [[\vp_i, \vp_i^{\dagger}], s] \rangle + \langle s, [2\gamma(E)\Id_E, s] \rangle \\
					&= -4\langle s, \sum_{i=1}^n [\vp_i, [\vp_i^{\dagger}, s]] \rangle = -4\sum_j |[\vp_j^{\dagger}, s]|^2 = -2 |[\vp^{\dagger}, s]|^2.
				\end{split}
			\end{equation*}
			
			For (iii), consider $\vp$ as a holomorphic section of $\Omega^1_{B_R} \otimes \End(E)$. Let $d_A$ be the Chern connection of $H$. The Hermitian metric $H$ on $E$ and the metric $g$ on $B_R$ induce a metric on $\Omega^1_{B_R} \otimes \End(E)$ with a corresponding connection $\nabla_A$. The induced curvature on $\Omega^1_{B_R} \otimes \End(E)$ is denoted by $\hat{F}$. We have $\sqrt{-1}\Lambda \hat{F} \vp = [\sqrt{-1}\Lambda F_A, \vp] + \Ric(\vp)$, where $\Ric$ is the Ricci tensor of $g$.
			
			By the Weitzenböck formula, we obtain:
			\begin{equation*}
				\begin{split}
					\Delta |\vp|^2 &= -2|\nabla_A \vp|^2 + 2 \langle \sqrt{-1}\Lambda \hat{F} \vp, \vp \rangle \\
					&= -2|\partial_A \vp|^2 + 2 \langle \sqrt{-1}\Lambda F_A \vp, \vp \rangle - 2 \Ric(\vp, \vp) \\
					&= -2|\nabla_A \vp|^2 - |[\vp, \vp^{\dagger}]|^2 - 2\Ric(\vp, \vp).
				\end{split}
			\end{equation*}
		\end{proof}
	\end{subsection}
	
	\begin{subsection}{$L^{\infty}$ Estimates}
		In this subsection, we outline techniques developed by Taubes \cite{taubes2013compactness}, which use \eqref{Eq_integralbyparts} and the $L^2(X)$ norm of $\phi$ to control the $L^{\infty}(X)$ norm of $\phi$ when $\dim(X) \leq 4$. There are also generalizations that use a Moser iteration arguement without any restriction on the dimension due to Collins \cite{Collinspersonal}. 
		
		\begin{lemma}{\cite{Collinspersonal,taubes2013compactness}}
			\label{interation_L2Linfinity}
			Let $f$ be a non-negative function. Suppose there exist positive constants $R$ and $\delta$ such that $\Delta |f|^2 + |d|f||^2 \leq R |f|^2$ and $\int_X |f|^{2(1+\delta)} \leq C$. Then, we have $\int_X |f|^{(1+\delta)\frac{2n}{n-2}} \leq C_{\delta}$, where $C_{\delta}$ is a constant that depends on $\delta$.
		\end{lemma}
		
		\begin{proof}
			First, multiply both sides of the inequality by $|f|^{2\delta}$ and integrate over $X$. This yields
			\[
			\int_X (|f|^{2\delta} \Delta |f|^2 + |f|^{2\delta} |d|f||^2) \leq C \int_X |f|^{2+2\delta}.
			\]
			Next, we compute
			\[
			\begin{split}
				&\int_X |f|^{2\delta} \Delta |f|^2 + |f|^{2\delta} |d|f||^2 = (4\delta + 1) \int_X |f|^{2\delta} |d|f||^2, \\
				&\int_X |d|f|^{1+\delta}|^2 = (1+\delta)^2 \int_X |d|f||^2 |f|^{2\delta}.
			\end{split}
			\]
			Therefore, we obtain
			\[
			\int_X |d|f|^{1+\delta}|^2 \leq \frac{4\delta + 1}{(1+\delta)^2} \left(\int_X |f|^{2\delta} \Delta |f|^2 + |f|^{2\delta} |d|f||^2 \right).
			\]
			Since $\int_X |f|^{2(1+\delta)} + |d|f|^{1+\delta}|^2 \leq C_{\delta}$, by the Sobolev inequality $L^2_1 \hookrightarrow L^{\frac{2n}{n-2}}$, we obtain the desired estimate.
		\end{proof}
		
		\begin{proposition}{\cite{Collinspersonal,taubes2013compactness}}
			\label{Prop_Lpforplarge}
			For any solution $(A, \phi)$ to the Hitchin-Simpson equations, and for any $p > 1$, we have $\|\phi\|_{L^p(X)} \leq C\|\phi\|_{L^2(X)}$.
		\end{proposition}
		
		\begin{proof}
			By Lemma \ref{lemmalocalestimate} and the Kato inequality, we have $\Delta |\phi|^2 + 2|d|\phi||^2 \leq C|\phi|^2$. Integrating over $X$, we get $\int_X |d|\phi||^2 \leq C \int_X |\phi|^2$. 
			
			Define $\hat{\phi} := \frac{\phi}{\|\phi\|_{L^2(X)}}$. Using the Sobolev inequality, we obtain $\int_X |\hat{\phi}|^{\frac{2n}{n-2}} \leq C$. Let $\delta_0$ be such that $2 + 2\delta_0 = \frac{2n}{n-2}$. Applying Lemma \ref{interation_L2Linfinity}, we get $\int_X |\hat{\phi}|^{(1 + \delta_0)\frac{2n}{n-2}} \leq C$. For $k \geq 1$, define $2 + 2\delta_k = (1 + \delta_{k-1}) \frac{2n}{n-2}$. Iterating Lemma \ref{lemmalocalestimate}, we obtain $\int_X |\hat{\phi}|^{2\delta_k + 2} \leq C$. As $\lim_{k \to \infty} \delta_k = \infty$, we get the desired estimate.
		\end{proof}
		
		\begin{proposition}
			\label{Prop_L2controlLinfty}
			For any solution to the Hitchin-Simpson equations, we have 
			\[
			\sup_X |\phi| \leq C \|\phi\|_{L^2(X)}.
			\]
		\end{proposition}
		
		\begin{proof}
			For each point $x\in X$, let $B_R$ be a ball of radius $R$ centered at $x$, and let $r$ denote the distance function from $x$. Take a cutoff function $\beta(r)$, which vanishes on $\partial B_{2R}$ and equals one on $B_R$. Let $G_x$ be the Green's function of $\Delta$ at $x$. By \eqref{Eq_integralbyparts}, we have
			\[
			\int_X \left( \beta G_x \Delta |\phi|^2 + 2 \beta G_x |\nabla_A \phi|^2 + 2 \beta G_x |\phi \wedge \phi|^2 \right) \leq C \int_{B_{2R}} G_x |\phi|^2.
			\]
			Next, we compute
			\[
			\int_X \beta G_x \Delta |\phi|^2 = \int_X \Delta (\beta G_x) |\phi|^2 = |\phi|^2(x) + \int_{B_{2R}} \Theta |\phi|^2,
			\]
			where $\Theta = \langle \nabla \chi, \nabla G_x \rangle - \langle \Delta, \chi G_x \rangle$, and $\|\Theta\|_{L^{\infty}} \leq C R^{1-2n}$. Since $$\beta G_x |\nabla_A \phi|^2 + \beta G_x |\phi \wedge \phi|^2 \geq 0,$$ we obtain
			\[
			|\phi|^2(x) \leq R^{1-2n} \int_{B_{2R}} |\phi|^2 + \int_{B_{2R}} G_x |\phi|^2.
			\]
			As $|G_x| \leq R^{2-2n}$, applying Hölder’s inequality, we obtain $$\int_{B_{2R}} G_x |\phi|^2 \leq \left( \int_{B_{2R}} |G_x|^p \right)^{\frac{1}{p}} \left( \int_{B_{2R}} |\phi|^{2q} \right)^{\frac{1}{q}},$$ where $p = \frac{2n-1}{2n-2}$ and $q = 2n-1$. By Proposition \ref{Prop_Lpforplarge}, we obtain the desired bound.
		\end{proof}
		
		\begin{proposition}
			Let $(A_i, \phi_i)$ be a sequence of solutions to the Hitchin-Simpson equations, and let $r_i := \|\phi_i\|_{L^2(X)}$. Define $\hat{\phi}_i = r_i^{-1} \phi_i$. Then:
			\begin{itemize}
				\item [(i)] $\int_X |d|\hat{\phi}_i||^2 + |\hat{\phi}_i|^2 \leq C$,
				\item [(ii)] $|\phi_i|$ converges weakly in the $L^2_1$ topology and strongly in the $L^p$ topology for any $p < \infty$.
			\end{itemize}
		\end{proposition}
		
		\begin{proof}
			Integrating \eqref{Eq_integralbyparts}, we obtain $\int_X |d_{A_i} \hat{\phi}_i|^2 + |\hat{\phi}_i|^2 \leq C$. By the Kato inequality, we conclude (i). As $\hat{\phi}_i$ converges weakly to $\mu_{\infty}$ in the $L^2_1$ topology, by Sobolev embedding, we obtain strong convergence in $L^p$ for any $p < \frac{2n}{n-1}$. By Proposition \ref{Prop_L2controlLinfty}, the boundedness of $\sup |\hat{\phi}_i|$ implies strong convergence of $|\hat{\phi}_i|$ in $L^p$ for all $p$, completing (ii).
		\end{proof}
	\end{subsection}
		
\begin{subsection}{$\MC^0$ Estimate in Terms of Eigenvalues}
	For any $R > 0$ and $x\in X$, we define $B_R := \{z \in X \mid \dist(z,x) < R\}$ as the disc of radius $R$ centered at $x$. Let $(E,H)$ be the Hermitian metric, $(\bar{\pa}_E, \vp)$ be a Higgs bundle of rank $r$, and consider $(A, \phi = \vp + \vp^{\dagger})$ be the corresponding solution to the Hitchin-Simpson equations \eqref{Hitchin-Simpson} over $B_R$.
	
	\begin{lemma}{\cite[P. 27]{Simpson1992}}
		For any matrix $A$ with eigenvalues $\lambda_1, \dots, \lambda_r$, let $g(A) = \sum_{i=1}^r |\lambda_i|^2$. Then, we have $|[A, A^{\dagger}]|^2 \geq C(|A|^2 - g(A))$, where $C$ is a positive constant.
	\end{lemma}
	
	\begin{proof}
		We choose a unitary basis such that $A = P + N$, where $P$ is diagonal and $N$ is strictly upper triangular. Thus, $|A|^2 = |P|^2 + |N|^2 = g(A) + |N|^2$. We now present the argument from \cite[Lemma 2.22]{Wentworth2016}, which proves by contradiction that $|[N, N^{\dagger}]| \geq C |N|^2$. Suppose there exists a sequence of strictly upper triangular matrices $N_i$ with $|N_i| = 1$ and $\lim |[N_i, N_i^{\dagger}]| = 0$. By passing to a subsequence, we obtain a limiting matrix $N_{\infty}$ with $|N_{\infty}| = 1$ and $|[N_{\infty}, N_{\infty}^{\dagger}]| = 0$. If we denote by $a_1, \dots, a_r$ and $b_1, \dots, b_r$ the rows and columns of $N_{\infty}$, the condition $[N_{\infty}, N_{\infty}^{\dagger}] = 0$ implies $|a_i|^2 = |b_i|^2$. Starting with $a_1 = 0$, we get $b_1 = 0$, which implies $a_2 = 0$. Continuing this process, we find that $N_{\infty} = 0$, contradicting $|N_{\infty}| = 1$. Thus, the lemma holds.
	\end{proof}
	
	\begin{corollary}
		\label{Cor_communicatorestimate}
		Let $\lambda_1, \dots, \lambda_r$ be the distinct eigenvalues of $\vp$. Define $g(\vp) := \sum_{i=1}^r |\lambda_i|^2$. Then, there exists a constant $C$, independent of the Higgs bundle, such that $$|[\vp, \vp^{\dagger}]|^2 \geq C(|\vp|^2 - g(\vp))^2.$$
	\end{corollary}
	
	\begin{proof}
		This is a local statement. Near a point $x\in X$ within an open neighborhood $B_R$, we choose an orthonormal basis $dz_1, \dots, dz_n$ such that $\vp = \sum_{j=1}^r \vp_j dz_j$. Let $\lambda_1^j, \dots, \lambda_r^j$ be the eigenvalues of $\vp_j$, and define $g_j(\vp) = \sum_{i=1}^r |\lambda_i^j|^2$. Then, $g(\vp) = \sum_{j=1}^r g_j(\vp)$. Since $|\vp|^2 = \sum_{j=1}^r |\vp_j|^2$, we compute:
		\[
		\begin{split}
			|[\vp, \vp^{\dagger}]|^2 &\geq \sum_{i=1}^n |[\vp_i, \vp_i^{\dagger}]|^2 \geq C \sum_{i=1}^n (|\vp_i|^2 - g_i(\vp))^2 \\
			&\geq \frac{C}{n} \left( \sum_{i=1}^n (|\vp_i|^2 - g_i(\vp)) \right)^2 \geq C' (|\vp|^2 - g(\vp))^2.
		\end{split}
		\]
	\end{proof}
	
	\begin{proposition}{\cite{Simpson1992}}
		\label{eigenvalueestimate}
		Over $B_R$, let $(A,\phi=\vp+\vp^{\da})$ be a solution to the Hitchin-Simpson equation, let $M$ be the maximum of the eigenvalues of $\vp$ over $B_R$. Then, there exists a constant $C$ such that
		\[
		\max_{B_{\frac{R}{2}}} |\vp|_H \leq C(M + 1).
		\]
	\end{proposition}
	
	\begin{proof}
		By Lemma \ref{lemmalocalestimate}, we have
		\[
		\Delta |\vp|^2 \leq -|[\vp, \vp^{\dagger}]|^2 + C |\vp|^2.
		\]
		Over $B_R$, let $\lambda_1, \dots, \lambda_r$ be the eigenvalues of $\vp$, and define $g := \sum_{i=1}^r |\lambda_i|^2$, then $g\leq rM^2$. By Corollary \ref{Cor_communicatorestimate}, we obtain
		\[
		\Delta |\vp|^2 \leq -C' (|\vp|^2 - g)^2 + C|\vp|^2.
		\]
		Now, over $B_{\frac{R}{2}}$, suppose $|\vp|$ achieves its maximum at some point in $B_{\frac{R}{2}}$. Denote this maximum as $Q := \max_{B_{\frac{R}{2}}} |\vp|$. Then we obtain 
		$C'(Q^2-g)^2\leq CQ^2$. Then either $Q^2\leq g$ or $Q^2-g\leq C''Q$ for some constant $C''$. We have either $Q\leq 2C''$ or $Q^2\leq 2C''g$, and the claim follows as $g\leq rM^2$.
	\end{proof}
\end{subsection}
	\begin{subsection}{Comparing Different Norms of the Higgs Fields}
		Let $(A, \phi = \vp + \vp^{\dagger}, H)$ be a solution to the Hitchin-Simpson equations \eqref{Hitchin-Simpson}, and let 
		\[
		\kappa(\vp) = (p_1(\vp), \dots, p_r(\vp)) \in \MA_X = \oplus_{i=1}^r H^0(\Sym^i \Omega^{1,0}_X)
		\]
		be the image of $(\bar{\pa}_E, \vp)$ under the Hitchin morphism $\kappa$. In this section, we compare the size of the following norms: $\|\vp\|_{L^2(X)}$, $\|\vp\|_{L^{\infty}(X)}$, and $\max_{1 \leq k \leq r} \|p_k(\vp)\|^{\frac{1}{k}}_{L^2(X)}$.
		
		\begin{proposition}
			\label{Prop_eigenvalueestimate}
			We have the following estimates:
			\begin{itemize}
				\item [(i)] Let $\lambda$ be an eigenvalue of $\vp$, then $|\lambda| \leq |\vp|$.
				\item [(ii)] $\|p_k(\vp)\|^{\frac{1}{k}}_{L^2_l(X)} \leq C_l \|\phi\|_{L^2(X)}$, where the constant $C_l$ depends on the positive integer $l$.
			\end{itemize}
		\end{proposition}
		
		\begin{proof}
			Locally, let $\lambda$ be the eigenvalue with eigenvector $s$. We have $|\lambda|^2 \langle s, s \rangle \leq \langle \vp s, \vp s \rangle \leq |\vp|^2 |s|^2$, which implies (i). For (ii), since $p_k(\vp)$ are invariant polynomials of $\vp$, we have $\|p_k(\vp)\|_{L^2(X)} \leq |\vp|^k_{L^{\infty}(X)}$. Moreover, since $p_k(\vp) \in H^0(\Sym^k \Omega^{1,0}_X)$ are holomorphic sections, elliptic estimates give us
			\[
			\|p_k(\vp)\|_{L^2_l(X)} \leq C \|p_k(\vp)\|_{L^2(X)} \leq C |\vp|^k_{L^{\infty}(X)} \leq C' \|\phi\|^k_{L^2(X)},
			\]
			where the last inequality follows from Proposition \ref{Prop_L2controlLinfty}.
		\end{proof}
		
		\begin{lemma}
			\label{LemmaA_controlmaximal}
			Let $B_{2R} \subset X$ be a small open ball, and let $M$ denote the maximum of the eigenvalues of $\vp$ over $B_R$. Then, $M \leq C \max_k \|p_k(\vp)\|^{\frac{1}{k}}_{L^2(B_R)}$.
		\end{lemma}
		
		\begin{proof}
			Let $\lambda$ be an eigenvalue of $\vp$. Then, $\lambda$ satisfies the spectral cover equation $\lambda^r + p_1(\vp)\lambda^{r-1} + \dots + p_r(\vp) = 0$. Define $t := \max_{1 \leq k \leq r} \|p_k(\vp)\|^{\frac{1}{k}}_{L^2(B_R)}$. By elliptic estimates, we have $\|p_k(\vp)\|_{L^{\infty}(B_R)} \leq \|p_k(\vp)\|_{L^2(B_R)}$. Thus, we obtain $|\lambda|^r \leq C (t |\lambda|^{r-1} + \dots + |\lambda| t^{r-1} + t^r)$. Therefore, $|\lambda|(|\lambda|^{r-1} - C t |\lambda|^{r-2} - \dots - C |\lambda| t^{r-1}) \leq t^r$, which implies $|\lambda| \leq C' t$.
		\end{proof}
		
		In summary, we can compare different norms of $\vp$:
		
		\begin{corollary}
			\label{Cor_RelationshipL2pk}
			Let $B_{2R}$ be an open ball in $X$ with radius $2R$ smaller than the injectivity radius. Let $M$ be the maximum of the eigenvalues of $\vp$ over $B_R$. Then:
			\begin{itemize}
				\item [(i)] $M \leq C \max_k \|p_k(\vp)\|^{\frac{1}{k}}_{L^{\infty}(B_R)}$,
				\item [(ii)] $\max_k \|p_k(\vp)\|^{\frac{1}{k}}_{L^2(B_R)} \leq C \|\vp\|_{L^2(X)}$,
				\item [(iii)] $\|\vp\|_{L^2(X)} \leq C (\max_k \|p_k(\vp)\|^{\frac{1}{k}}_{L^2(B_R)} + 1)$.
			\end{itemize}
		\end{corollary}
		
		The following result is due to Tanaka \cite[Prop 2.7]{tanaka2019singular} for $\mathrm{SL}_2(\mathbb{C})$ Higgs bundles over Kähler surfaces. We extend his result and prove:
		
		\begin{theorem}
			\label{Theorem_convergencespectralcover}
			Let $(A_i, \phi_i = \vp_i + \vp_i^{\dagger})$ be a sequence of solutions to the Hitchin-Simpson equations. Define $r_i := \|\phi_i\|_{L^2(X)}$ and set $\hat{\vp}_i = r_i^{-1} \vp_i$. Then:
			\begin{itemize}
				\item [(i)] Suppose $r_i \leq C$. Then, $p_k(\vp_i)$ converges in the $\MC^{\infty}$ topology over $X$.
				\item [(ii)] Suppose $\limsup r_i = \infty$. Then, passing to a subsequence, we have:
				\begin{itemize}
					\item [(a)] $p_k(\hat{\vp}_i)$ converges in the $\MC^{\infty}$ topology over $X$ to $\alpha_k$, where $\alpha_k \in H^0(\Sym^k \Omega^{1,0}_X)$,
					\item [(b)] $(\alpha_1, \dots, \alpha_r) \neq 0$.
				\end{itemize}
			\end{itemize}
		\end{theorem}
		
		\begin{proof}
			Part (i) and (a) follow directly from Proposition \ref{Prop_eigenvalueestimate}. For (b), by Corollary \ref{Cor_RelationshipL2pk}, we have
			\[
			\|\phi_i\|_{L^2(X)} \leq C \left( \sum_{k=1}^r \|p_k(\vp_i)\|^{\frac{1}{k}}_{L^2(X)} + 1 \right).
			\]
			Therefore,
			\[
			\sum_{k=1}^r \|p_k(\hat{\vp}_i)\|^{\frac{1}{k}}_{L^2(X)} + r_i^{-1} \geq C^{-1}.
			\]
			As $\limsup r_i = \infty$, we obtain $\sum_{k=1}^r \|\alpha_k\|^{\frac{1}{k}}_{L^2(X)} \geq C^{-1} > 0$, which implies (ii).
		\end{proof}
	\end{subsection}
	\end{section}
	
	\begin{section}{Local Estimates For Projections Outside of the Discriminant Locus}
\label{section_localestimate}
In this section, we present estimates for solutions outside the discriminant locus. Our approach is largely inspired by the work of Mochizuki \cite{Mochizukiasymptotic}, who developed similar estimates for Riemann surfaces.

\begin{subsection}{Local Decomposition for Higgs Bundles}
	Let $(E, H)$ be a Hermitian bundle with a Higgs field $(\bar{\pa}_E, \vp)$, and let $(A, \phi = \vp + \vp^{\dagger})$ be the corresponding pair that solves the Hitchin-Simpson equations. We assume that $(\bar{\pa}_E, \vp)$ is generically semisimple.
	
	Let $Z_{\vp}$ denote the discriminant locus of $\vp$. For $B_R \subset X \setminus Z_{\vp}$, a ball of radius $R$ centered at $x$, the restriction $\vp|_{B_R}$ has $r$ distinct eigenvalues, giving a decomposition of the Higgs bundle as
	\begin{equation}
		\label{eq_decomposition_Higgs_bundles}
		(E, \bar{\pa}_E, \vp) = \bigoplus_{i=1}^r (E_i, \bar{\pa}_i, \vp_i),
	\end{equation}
	where $\rank(E_i) = 1$, and $\vp_i = \lambda_i \Id_{E_i}$, with $\lambda_i$ a holomorphic section of $\Omega_X^{1,0}$. Note that this decomposition is not orthogonal with respect to the Hermitian metric.
	
	We define the rescaled Higgs field as $\hvp := \frac{\vp}{\|\phi\|_{L^2(X)}}$ and denote its image under the Hitchin morphism as $\mbs_{\hvp} := \kappa(\hvp) \in \MA_X = \bigoplus_{i=1}^r H^0(\Sym^i \Omega_X^{1,0})$. By definition, the discriminant locus $Z_{\vp}$ of $\vp$ coincides with the discriminant locus $Z_{\hvp}$ of $\hvp$.
	
	\begin{proposition}
		\label{Prop_higgsbundlesemisimpledecompositionestimate}
		For every $x\in X \setminus Z_{\vp}$, define $d_x := \min_{i \neq j} |\lambda_i - \lambda_j|(x)$. There exists a radius $R > 0$ such that over $B_R$:
		\begin{itemize}
			\item $|\vp| \leq C C_0 d_x + C$,
			\item $|\lambda_i - \lambda_i(x)| \leq \frac{d_x}{100}$,
		\end{itemize}
		where $R$ and $C_0$ are constants depending only on $\mbs_{\hvp}$. Furthermore, if we denote $M_x := \max_i \lambda_i(x)$, then $R = \min\{R_0, \frac{d_x}{100 C \|\phi\|_{L^2(X)}}\}$ and $C_0 = \frac{M_x}{d_x}$, where $R_0$ is the injectivity radius of $X$.
	\end{proposition}
	
	\begin{proof}
		Let $M$ denote the maximum of the eigenvalues of $\vp$ over $B_R$. By Proposition \ref{Prop_eigenvalueestimate}, we obtain:
		\[
		|M - M_x| \leq C \|\phi\|_{L^2(X)} R, \quad |\lambda_i - \lambda_i(x)| \leq C \|\phi\|_{L^2(X)} R.
		\]
		We take $C_0 = \frac{M_x}{d_x}$ and $R = \min\left\{\frac{M_x}{C \|\phi\|_{L^2(X)}}, \frac{d_x}{100 C \|\phi\|_{L^2(X)}}\right\}$. Then, $M \leq 2M_x$ and $|\lambda_i - \lambda_i(x)| \leq \frac{d_x}{100}$. By Proposition \ref{eigenvalueestimate}, we have
		\[
		\max_{B_R} |\vp| \leq C(M + 1) \leq 2C(M_x + 1) \leq 2C(C_0 d_x + 1).
		\]
		Additionally, since $d_x \leq 2M_x$, we have $\frac{d_x}{100 C \|\phi\|_{L^2(X)}} \leq \frac{2M_x}{C \|\phi\|_{L^2(X)}}$. Rewriting the constants, we obtain the desired result.
	\end{proof}
	
	\begin{remark}
		The significance of the constants in Proposition \ref{Prop_higgsbundlesemisimpledecompositionestimate} is that they depend only on the rescaled Higgs field $\hvp$, not on $\vp$. By Theorem \ref{Theorem_convergencespectralcover}, for an unbounded sequence, the spectral cover of the rescaled Higgs bundle converges. Therefore, constants depending on the spectral data of $\hvp$ are well controlled.
	\end{remark}
	
	\begin{corollary}
		Under the previous assumptions, over $B_R$, we have $|\lambda_i - \lambda_j| \geq \frac{48}{50} d_x$.
	\end{corollary}
	
	\begin{proof}
		By the definition of $d_x$, we have $|\lambda_i(x) - \lambda_j(x)| \geq d_x$. The statement follows from the inequality
		\[
		|\lambda_i(x) - \lambda_j(x)| \leq |\lambda_i - \lambda_i(x)| + |\lambda_i - \lambda_j| + |\lambda_j - \lambda_j(x)| \leq \frac{1}{50} d_x + |\lambda_i - \lambda_j|.
		\]
	\end{proof}
\end{subsection}
\begin{subsection}{Estimate on the Projection}
	For the decomposition \eqref{eq_decomposition_Higgs_bundles}, let $\pi_i: (E, \bar{\pa}_E) \to (E_i, \bar{\pa}_i)$ be the holomorphic projection onto each line bundle. Using the Hermitian metric, we define the orthogonal projection $\pi_i': (E, \bar{\pa}_E) \to (E_i, \bar{\pa}_i)$. In this subsection, we derive estimates and analyze the difference between these two projections.

	\begin{lemma}
		\label{Lemma_factsabouttheseprojection}
			Define $\chi_i := \pi_i - \pi_i'$, then the following hold:
		\[
		\bar{\pa}_E \pi_i = 0, \quad [\vp, \pi_i] = 0, \quad \pi_i'^{\dagger} = \pi_i', \quad |\pi_i'| = 1, \quad |\pi_i|^2 = 1 + |\chi_i|^2.
		\]
	\end{lemma}
	
	\begin{proof}
		The relations $\bar{\pa}_E \pi_i = 0$ and $[\vp, \pi_i] = 0$ follow from the fact that the decomposition \eqref{eq_decomposition_Higgs_bundles} is holomorphic. The properties $\pi_i'^{\dagger} = \pi_i'$ and $|\pi_i'| = 1$ follow from the definition of $\pi_i'$, which respects the Hermitian metric. Note that for $s\in E_i$, we have $\chi_is=0$ and for $s\in E_i^{\perp}$, $\pi_i's=0$. Therefore, we have $\langle \pi_i', \chi_i \rangle = 0$, which implies $|\pi_i|^2 = |\pi_i'|^2 + |\chi_i|^2$.
	\end{proof}
	
	The following pointwise linear algebra lemma is widely used by Mochizuki in \cite{mochizuki2003asymptoticdifferent, mochizuki2008wild}.
	
	\begin{lemma}{\cite[Lemma 2.8]{Mochizukiasymptotic}}
		\label{Lemma_Vectorspaceprojection}
		Let $V$ be a vector space of rank $r$ with a Hermitian metric $h$, and let $\theta$ be an endomorphism of $V$. Suppose there exists a decomposition $(V, \theta) = \oplus_{i=1}^r (V_i, \theta_i)$, where $\rank(V_i) = 1$. Suppose further that the following data hold:
		\begin{itemize}
			\item [(i)] $\alpha_1, \dots, \alpha_r$ are $r$ complex numbers, and $d := \min_{i \neq j} |\alpha_i - \alpha_j|>0$,
			\item [(ii)] $|\theta|_h \leq C_0(d + 1)$,
			\item [(iii)] $|\theta_i - \alpha_i| \leq \frac{d}{100}$.
		\end{itemize}
		Let $\pi_i$ be the projection of $V$ onto $V_i$, and let $\pi_i'$ be the orthogonal projection of $V$ onto $V_i$. Define $\chi_i := \pi_i - \pi_i'$. Then:
		\begin{itemize}
			\item [(i)] There exists a constant $C$ depending only on $r$ such that $|\pi_i|_h \leq C$ and $|\chi_i| \leq C$.
			\item [(ii)] There exists a constant $C'$ depending only on $r$ such that $|[\theta^{\dagger}, \pi_i]|_h \geq C' d |\chi_i|$.
		\end{itemize}
	\end{lemma}
	
	\begin{proposition}
		\label{Prop_adjointwithprojectionestimate}
		Under the previous notation, for $x\in X \setminus Z_{\vp}$. There exist constants $R$, $C$, and $C'$, depending on $r$ and $\mbs_{\hvp}$, such that over $B_R$, we have $$|\pi_i| \leq C,\;|\chi_i| \leq C',\;|[\vp^{\dagger}, \pi_i]| \geq C' d_x |\chi_i|.$$
	\end{proposition}
	
	\begin{proof}
		For the decomposition \eqref{eq_decomposition_Higgs_bundles}, we work locally in $B_R$. Choose an orthogonal basis $dz_1, \dots, dz_n$ for $T^* B_R$, and write $\vp_i = \sum_{j=1}^n \vp_i^j dz_j$. Also, write $\lambda_i = \sum_{j=1}^n \gamma_i^j dz_j$ with $\vp_{i}^j=\gamma_i^j\Id_{E_i}$ and $\vp = \sum_{j=1}^n \psi^j dz_j$ with $\psi^j = \sum_{i=1}^r \vp_i^j$.
		
		Define $d_{x,j} := \min_{i \neq i'} |\gamma_i^j - \gamma_{i'}^j|(x) \geq 0$, which measures the difference of the eigenvalues under the frame $dz_j$. By definition, $d_x = \sum_{j=1}^n d_{x,j}$. Let $j_0$ such that $d_{x,j_0} = \max_j \{d_{x,j}\}$, so that $d_{x,j_0} \leq d_x \leq n d_{x,j_0}$. As $d_x>0$, we have $d_{x,j_0}>0.$
		
		By Proposition \ref{Prop_higgsbundlesemisimpledecompositionestimate}, we have:
		\[
		|\psi^{j_0}|\leq |\vp| \leq C_0 (d_x + 1) \leq n C_0 (d_{x,j_0} + 1),
		\]
		and
		\[
		|\gamma_i^{j_0} - \gamma_i^{j_0}(x)| \leq |\lambda_i - \lambda_i(x)| \leq \frac{d_x}{100} \leq \frac{n d_{x,j_0}}{100}.
		\]
		Applying Lemma \ref{Lemma_Vectorspaceprojection} to $\psi^{j_0}$ with $d = d_{x,j_0}$ and $\alpha_i = \gamma^{j_0}_i(x)$, we obtain constants $C, C'$ such that $|\pi_i| \leq C$, $|\chi_i| \leq C'$, and $|[(\psi^{j_0})^{\dagger}, \pi_i]| \geq C' d_{x,j_0} |\chi_i|$. Therefore,
		\[
		|[\vp^{\dagger}, \pi_i]|^2 = 2 \sum_{j=1}^n |[(\psi^j)^{\dagger}, \pi_i]|^2 \geq |[(\psi^{j_0})^{\dagger}, \pi_i]|^2 \geq (C')^2 d_{x,j_0}^2 |\chi_i|^2 \geq \frac{(C')^2}{n^2} d_x^2 |\chi_i|^2.
		\]
	\end{proof}
	
	\begin{lemma}
		\label{Lemma_Alforslemma}
		Let $f \geq 0$ be a smooth function over $B_R$ such that $f \leq C_1$ and $\Delta f \leq -d^2 f$. Then, $|f| \leq C_1 \exp(-C_2 d (R^2 - r^2))$, where $C_2 = \min\left\{\frac{d}{4n}, \frac{1}{4R}\right\}$.
	\end{lemma}
	
	\begin{proof}
		Let $r$ be the radial distance to the center of $B_R$, and consider a comparision function $h(r):=e^{C_2d_xr^2}$ such that for suitable choice of $C_2$, we have $\Delta h(r)\geq -d^2h(r)$. Define $g := C_1 h(r)$, which satisfies $g|_{\partial B_R} = C_1 \geq f$. Then,
		\[
		\Delta(f - g) \leq d^2 (f - g) \quad \text{over} \ B_R, \quad (f - g)|_{\partial B_R} \leq 0.
		\]
		By the maximum principle, $f \leq g$ over $B_R$.
	\end{proof}
	
By Proposition \ref{Prop_higgsbundlesemisimpledecompositionestimate}, we have the following proposition:
	
	\begin{proposition}
		\label{prop_estimateonchi}
		Under the previous conventions, let $r$ denote the distance to $p$. Over $B_R$, we have $|\chi_i| \leq C e^{-C' d_x (R^2 - r^2)}$, where $C, C'$ depend on $\mbs_{\hvp}$.
	\end{proposition}
	
	\begin{proof}
		By Lemma \ref{Lemma_conjugatevanishing}, we have:
		\[
		\Delta |\pi_i|^2 + 2|\partial_A \pi_i|^2 + 2|[\vp^{\dagger}, \pi_i]|^2 = 0.
		\]
		By Proposition \ref{Prop_adjointwithprojectionestimate}, we have $|[\vp^{\dagger}, \pi_i]| \geq B' d_x |\chi_i|$ and $|\chi_i| \leq B$, with $B, B'$ constants depending on the spectral data $\mbs_{\hvp}$. Moreover, by Lemma \ref{Lemma_factsabouttheseprojection}, $|\pi_i|^2 = |\chi_i|^2 + 1$.
		
		Therefore, we obtain:
		\[
		\Delta |\chi_i|^2 \leq -B^2 d_x^2 |\chi_i|^2, \quad \text{with} \ |\chi_i| \leq B'.
		\]
		Applying Lemma \ref{Lemma_Alforslemma}, we get $|\chi_i|^2 \leq B'^2 e^{-C' d_x (R^2 - r^2)}$, with $C' = \min\left\{\frac{B d}{4n}, \frac{\sqrt{B}}{4R}\right\}$.
	\end{proof}
	
	\begin{corollary}
		\label{corocommutevanishofpi}
		Under the previous assumptions, we have the following estimates over the half-sized ball $B_{\frac{R}{2}}$:
		\begin{itemize}
			\item [(i)] $|\pi_i-\pi_i^{\da}|\leq Ce^{-C'd_x}$
			\item [(ii)] $|[\vp, \pi_i^{\dagger}]| \leq C e^{-C' d_x}$ over $B_{\frac{R}{2}}$.
			\item [(iii)] Let $s_i, s_j$ be local sections of $\ME_i$ and $\ME_j$, respectively. Then,
			\[
			|H(s_i, s_j)| \leq C e^{-C' d_x} |s_i|_H |s_j|_H,
			\]
			where $C, C'$ depend on $\mbs_{\hvp}$.
		\end{itemize}
	\end{corollary}
	
	\begin{proof}
		As $|\pi_i-\pi_i^{\da}|=|\pi_i-\pi_i'+\pi_i'-\pi_i^{\da}|\leq |\chi_i|+|\chi_i^{\da}|$, by Proposition \ref{prop_estimateonchi}, we obtain (i). 
		
		As $[\vp,\pi_i]=0$, we have $|[\vp,\pi_i^{\da}]|=|[\vp,\pi_i-\pi_i^{\da}]|\leq |\vp||\pi_i-\pi_i^{\da}|$. By Proposition \ref{Prop_higgsbundlesemisimpledecompositionestimate} and \ref{prop_estimateonchi}, over $B_{\frac{R}{2}}$, we have $|\vp| \leq C(d_x + 1)$. Then by (i) and adjust the constants, we obtain (ii).
		
		For (iii), we compute $H(s_i, s_j) = H(\pi_i s_i, s_j) = H(s_i, (\pi_i^{\dagger} - \pi_i) s_j)$. Since $\pi'^{\dagger} = \pi'$, we have $\pi_i^{\dagger} - \pi_i = \chi_i^{\dagger} - \chi_i$. By Proposition \ref{prop_estimateonchi}, we obtain the desired estimate.
	\end{proof}
	
	\begin{proposition}
		\label{Theoremcurvatureestimate}
		Under the previous assumptions, over $B_{\frac{R}{2}}$, we have $|[\vp, \vp^{\dagger}]| \leq C e^{-C' d_x}$. In particular, $|\Lambda F_A^{\perp}| = |\Lambda [\vp, \vp^{\dagger}]| \leq C e^{-C' d_x}$, where $C, C'$ depend on $\mbs_{\hvp}$.
		
		Moreover, if $\deg(E) = 0$ and $\ch_2(E) \cdot [\omega]^{n-2} = 0$, then $|F_A| \leq C e^{-C' d_x}$.
	\end{proposition}
	
	\begin{proof}
		We write $\vp^{\dagger} = \sum_{i,j=1}^r \pi_i \circ \vp^{\dagger} \circ \pi_j$. Then,
		\[
		|[\vp, \vp^{\dagger}]| = \sum_{k,i,j} |[\vp_k, \pi_i \circ \vp^{\dagger} \circ \pi_j]|.
		\]
		By Corollary \ref{corocommutevanishofpi}, for $i \neq j$, we have $|\pi_i \circ \vp^{\dagger} \circ \pi_j| = |[\pi_i, \vp^{\dagger}] \circ \pi_j| \leq C e^{-C d_x}$. Additionally, for $i \neq k$, we have $[\vp_k, \pi_i \circ \vp^{\dagger} \circ \pi_i] = 0$. Since $\rank(\ME_i) = 1$, we get $[\vp_k, \pi_k \circ \vp^{\dagger} \circ \pi_k] = 0$. Therefore, we obtain the desired estimate for $[\vp, \vp^{\dagger}]$. The estimate for $|\Lambda F_A^{\perp}|$ follows directly from the Hitchin-Simpson equations \eqref{Hitchin-Simpson}.
		
		Furthermore, if $\deg(E) = 0$ and $\ch_2(E) \cdot [\omega]^{n-2} = 0$, then
		\[
		|F_A| \leq |[\vp, \vp^{\dagger}]| \leq C e^{-C' d_x}.
		\]
	\end{proof}
	
	\begin{proposition}
		Over $B_R$, we have $\|d_A^{*} F_A\|_{L^p} \leq C e^{-C' d_x} \|\nabla_A \phi\|_{L^p}$, where $C, C'$ depend on $\mbs_{\hvp}$.
		\label{Prop_curvaturederivativeestimates}
	\end{proposition}
	
	\begin{proof}
		We compute $d_A^{*} F_A = i \partial_A \Lambda F_A - i \bar{\partial}_A \Lambda F_A$. Since $|\partial_A \Lambda F_A| = |\bar{\partial}_A \Lambda F_A|$, it suffices to estimate $\partial_A \Lambda F_A$. Write $\vp = \sum_i\vp_i dz_i$, $\vp^{\dagger} = \sum_i\vp_i^{\dagger} d\bar{z}_i$, and $\bar{\partial}_A = \sum_i\bar{\partial}_i d\bar{z}_i$.
		
		We compute:
		\begin{equation*}
			\begin{split}
						|\partial_A \Lambda F_A| = |\partial_A \Lambda [\vp, \vp^{\dagger}]| \leq& |\Lambda \partial_A [\vp, \vp^{\dagger}]| + |\bar{\partial}_A^{*} [\vp, \vp^{\dagger}]| \\
						\leq& |\Lambda [\partial_A \vp, \vp^{\dagger}]| + \sum_{i=1}^r |[\bar{\partial}_i \vp_i, \vp_i^{\dagger}]| + \sum_{i=1}^r |[\vp_i, \bar{\partial}_i \vp_i^{\dagger}]|.
			\end{split}
		\end{equation*}

		Write $\vp^{\dagger} = \sum_{i,j=1}^r \pi_i \circ \vp^{\dagger} \circ \pi_j$. For $i \neq j$, by Corollary \ref{corocommutevanishofpi}, we have $$|\pi_i \circ \vp^{\dagger} \circ \pi_j| = |[\pi_i, \vp^{\dagger}] \circ \pi_j| \leq C e^{-C' d_x},$$ and $[\pi_i \circ \partial_A \vp, \pi_j \circ \vp^{\dagger} \circ \pi_j] = 0$ for $i \neq j$.
		
		Additionally, since $\rank(\ME_i) = 1$, we have $[\pi_i \circ \partial_A \vp, \pi_i \circ \vp^{\dagger} \circ \pi_i] = 0$. Therefore, we obtain:
		\[
		|[\partial_A \vp, \vp^{\dagger}]| = |[\pi_i \partial_A \vp, \pi_i \circ \vp^{\dagger} \circ \pi_j]| \leq C e^{-C' d_x} |\partial_A \vp|.
		\]
		Similarly, we have the estimate $|[\bar{\partial}_i \vp_i, \vp_i^{\dagger}]| \leq C e^{-C' d_x} |\bar{\partial}_i \vp_i|$. Thus, we conclude:
		\[
		\|d_A^{*} F_A\|_{L^p} \leq C e^{-C' d_x} \|\nabla_A \phi\|_{L^p}.
		\]
	\end{proof}
\end{subsection}

\subsection{Decay Estimates on Line Bundles}
Let $(E,H)$ be a Hermitian vector bundle, and let $(A, \phi)$ be a solution to the Hitchin-Simpson equations with $(\bar{\pa}_E, \vp)$ corresponding to a Higgs bundle, assumed to be generically semisimple. For the decomposition \eqref{eq_decomposition_Higgs_bundles} over $B_R$, we study the decay estimate on the line bundles for the restriction of the harmonic bundle.

\subsubsection{Uhlenbeck Gauge Fixing}
We begin by introducing the Uhlenbeck gauge fixing condition. A unitary connection $A$ is said to have finite energy if $\int_X |F_A|^2 < \infty$, and $A$ is integrable if $F_A^{(2,0)} = 0$. The following result holds:

\begin{theorem}{\cite{uhlenbeck1982connections, uhlenbeck1986aprior}}
	\label{Uhlenbeckcompactnessforlocalmodel}
	Let $E \to X$ be a complex vector bundle with Hermitian metric $H$. For any $p > n$, there exist $R_0 > 0$ and $\epsilon_0 > 0$ such that for an integrable connection $A$ on $E$, if the following conditions hold over a radius-$2R$ ball $B_{2R} \subset X$:
	\begin{itemize}
		\item [(i)] $\|\Lambda F_A\|_{L^{\infty}(B_{2R})} < \epsilon_0$,
		\item [(ii)] $R^{4-2n} \int_{B_{2R}} |F_A|^2 < \epsilon_0$,
		\item [(iii)] $R \leq \min(\epsilon_0, R_0)$,
	\end{itemize}
	then there exists a local trivialization of $E|_{B_{2R}} \cong B_{2R} \times \mathbb{C}^r$ such that, writing $d_A = d + A$, we have the estimate:
	\begin{equation}
		\label{eq_Uhlenbeckestimate}
		R^{2 - \frac{2n}{p}} \|\nabla A\|_{L^p(B_R)} + R^{1 - \frac{2n}{p}} \|A\|_{L^p(B_R)} \leq C R^{2-n} \|F_A\|_{L^2(B_{2R})} + C R^{\frac{2n}{p}} \|\Lambda F_A\|_{L^{\infty}(B_{2R})}.
	\end{equation}
\end{theorem}
The constant $\ep_0$ is usually called the Uhlenbeck constant. 

\subsubsection{First Derivative Estimates}
Let $(E,H)$ be the Hermitian vector bundle with $(A,\phi)$ a solution to \eqref{Hitchin-Simpson}.  For the decomposition \eqref{eq_decomposition_Higgs_bundles}, we have the holomorphic projection $\pi_i: E \to E_i$ and the orthogonal projection $\pi_i': E \to E_i$. We will show an almost orthogonality for these projections at the level of first derivatives, which might requires an adjustment of the Uhlenbeck constant $\ep_0$.

\begin{proposition}
	\label{estimateontheporjection}
	Suppose over $B_{2R}$ there exist a small constant $\epsilon_1$ and a large constant $d_1$ such that the Chern connection $d_A=\bar{\pa}_A+\pa_A$ satisfies the Uhlenbeck gauge fixing condition in Theorem \ref{Uhlenbeckcompactnessforlocalmodel} with $R^{2-n} \|F_A\|_{L^2(B_{2R})} \leq \epsilon_1$ and $d_x \geq d_1$. Then there exist positive constants $C, C'$ depends on $\mbs_{\hvp}$, $d_1$ and $\ep_1$, such that over $B_R$, we have:
	\begin{equation}
		\begin{split}
			|\partial_A \pi_i| = |\bar{\partial}_A \pi_i^{\dagger}| \leq C e^{-C' d_x}, \\
			|\partial_A \pi_i'| = |\bar{\partial}_A \pi_i'| \leq C e^{-C' d_x},
		\end{split}
	\end{equation}
	where the constants depend on $\mbs_{\hat{\vp}}$.
\end{proposition}

\begin{proof}
	We begin by estimating the derivative of $\pi_i$. Since $\bar{\partial}_A \pi_i = 0$, Theorem \ref{Theoremcurvatureestimate} implies that:
	\[
	|\Lambda \bar{\pa}_A \partial_A \pi_i| = |[\Lambda F_A, \pi_i]| \leq C e^{-C' d_x}.
	\]
	Since $\partial_A (\pi_i^{\dagger}) = 0$, we also have $|\Lambda\bar{\pa}_A \partial_A (\pi_i - \pi_i^{\dagger})| \leq C e^{-C' d_x}$. Moreover, by Corollary \ref{corocommutevanishofpi}, $|\pi_i - \pi_i^{\dagger}| \leq C e^{-C' d_x}$.
	
	We define $\nu:=\pi_i - \pi_i^{\dagger}$, using Theorem \ref{Uhlenbeckcompactnessforlocalmodel}, in the local trivialization $E|_{B_{2R}} \cong B_{2R} \times \mathbb{C}^r$, we write $d_A = d + A$ and apply the estimate \eqref{eq_Uhlenbeckestimate}. The covariant derivative with respect to the trivial connection can be schematically written as:
	\begin{equation}
		\begin{split}
		\label{eq_L2pestimate}
		\Lam\bar{\pa}_A\pa_A\nu=&\Lam(\bar{\pa}+A^{0,1})(\pa+A^{1,0})\nu\\
		=&\Lam\bar{\pa}\pa\nu+\Lam([A^{0,1},\pa\nu]+[A^{1,0},\bpa \nu])+\Lam[\bar{\pa}A^{1,0},\nu]+\Lam[A^{0,1},[A^{1,0},\nu]].
	\end{split}
	\end{equation}
	
	WLOG, we could rescale and assume $R=1$ and do the estimates on a unit size ball $B_1$. By Proposition \ref{Prop_curvaturederivativeestimates}, for $p>n$, we have $\MC^0\hookrightarrow L_1^p$. By \eqref{eq_Uhlenbeckestimate}, we have 
	\begin{equation}
		|[A^{0,1},[A^{1,0},\nu]]| \leq |A|^2 |\nu|\leq C(C_{\ep_0}+e^{-C'd_x})e^{-C'd_x},
	\end{equation}
	and 
	\begin{equation}
		\|[\bar{\pa}A^{1,0},\nu]\|_{L^p}\leq \|\na A\|_{L^{2p}}\|\nu\|_{L^{2p}}\leq C(C_{\ep_0}+e^{-C'd_x})e^{-C'd_x},
	\end{equation}
where $C_{\ep_0}$ is a constant depends on $\ep_0$.
	
	For $p>\frac{n}{2}$, we have $L^2_p\hookrightarrow L_1^{2p}$ and we have:
	\begin{equation}
		\begin{split}
			&\|[A^{0,1},\pa\nu]\|_{L^p}+\|[A^{1,0},\bar{\pa}\nu]\|_{L^p}\leq \|A\|_{L^{2p}}\cdot \|\na \nu\|_{L^{2p}}\\
			\leq &C(\|F_A\|_{L^2}+e^{-C'd})\| \nu\|_{L^p_2}\leq  C(\ep_1+e^{-C'd_x})\| \nu\|_{L^p_2}.
		\end{split}
	\end{equation}
By elliptic of $\Lam\bar{\pa}\pa$, we have $\|\nu\|_{L^p_2}\leq C(\|\Lam\bar{\pa}\pa\nu\|_{L^p}+\|\nu\|_{L^p})$. Combining with \eqref{eq_L2pestimate}, and previous estimates, we have 
\begin{equation}
	\begin{split}
		\|\nu\|_{L^p_2}\leq C(\ep_1+e^{-C'd_x})\|\nu\|_{L^p_2}+C_{\ep_0}e^{-C'd_x}.
	\end{split}
\end{equation}
	
Choosing $\epsilon_1$ and $d_1$ such that $C \epsilon_1 \leq \frac{1}{4}$ and for $d_x>d_1$, we have $C e^{-C' d_x} \leq \frac{1}{4}$. Therefore, for any $p> n$, we have
	\begin{equation}
		\|\pi_i - \pi_i^{\dagger}\|_{L_{2}^p}=\|\nu\|_{L_2^p}\leq C_{\ep_1} e^{-C' d_x},
	\end{equation}
where $C_{\ep_1}$ depends on the choice of $\ep_1$. Moreover, for $p>n$, we have
$$
|\pa_A\nu|\leq |\pa\nu|+|[A^{1,0},\nu]|\leq C(\|\nu\|_{L_2^p}+|\nu|)\leq Ce^{-C'd_x}.
$$
As $|\pa_A\nu|=|\partial_A (\pi_i - \pi_i^{\dagger})| = |\partial_A \pi_i|$, we obtain the desire estimate $|\pa_A\pi_i|\leq Ce^{-C'd_x}$.
	
	We now estimate $|\partial_A \pi_i'|$. Let $E = E_i \oplus E_i^{\perp}$ be the orthogonal decomposition using the harmonic metric $H$. Since $\chi_i := \pi_i - \pi_i'$ satisfies $\chi_i|_{E_i} = 0$, we can regard $\chi_i$ as a morphism $\chi_i: E_i^{\perp} \to E_i$, and similarly, $\chi_i^{\dagger} = \pi_i^{\dagger} - \pi_i: E_i \to E_i^{\perp}$.
	
	Since $(E_i,\pa_i)$ is a holomorphic subbundle, it has an induced holomorphic structure $\bar{\partial}_{i^{\perp}}$ on $E_i^{\perp}$, and we write $\bar{\partial}_A^0 := \bar{\partial}_{i} \oplus \bar{\partial}_{i^{\perp}}$ as the holomorphic structure on $E_i \oplus E_i^{\perp}$. Then, we could write $\bar{\partial}_A = \bar{\partial}_A^0 + \mfe$, where $\mfe \in \Omega^{0,1} \otimes \Hom(E_i^{\perp}, E_i)$. Since $\mfe|_{E_i} = \chi_i|_{E_i} = 0$, we have $[\mfe, \chi_i] = 0$. We compute:
	\[
	\bar{\partial}_A (\pi_i - \pi_i^{\dagger}) = \bar{\partial}_A \chi_i - \bar{\partial}_A \chi_i^{\dagger} = \bar{\partial}_A^0 \chi_i - \bar{\partial}_A^0 \chi_i^{\dagger} - \mfe \circ \chi_i^{\dagger} - \chi_i^{\dagger} \circ \mfe.
	\]
	
	Noting that:
	\begin{equation*}
		\begin{split}
			\bar{\partial}_A^0 \chi_i \in \Omega^{0,1} \otimes \Hom(E_i^{\perp}, E_i), \quad \bar{\partial}_A^0 \chi_i^{\dagger} \in \Omega^{0,1} \otimes \Hom(E_i, E_i^{\perp}), \\
			\mfe \circ \chi_i^{\dagger} \in \Omega^{0,1} \otimes \Hom(E_i, E_i), \quad \chi_i^{\dagger} \circ \mfe \in \Omega^{0,1} \otimes \Hom(E_i, E_i),
		\end{split}
	\end{equation*}
	and these are orthogonal with respect to $H$, we have $|\bar{\partial}_A^0 \chi_i| \leq |\bar{\partial}_A (\pi_i - \pi_i^{\dagger})|_H = |\bar{\partial}_A \pi_i^{\dagger}|$. Additionally, since $\bar{\partial}_A \pi_i' = -\bar{\partial}_A \chi_i = -\bar{\partial}_A^0 \chi_i$, we obtain:
	\[
	|\bar{\partial}_A \pi_i'| \leq |\bar{\partial}_A \pi_i^{\dagger}|\leq Ce^{-C'd_x}
	\]
	which proves the claim.
\end{proof}

If $\deg(E) = 0$ and $\ch_2(E) \cdot [\omega]^{n-2} = 0$, then the assumption $R^{2-n} \|F_A\|_{L^2(B_{2R})} \leq \epsilon_1$ in Proposition \ref{estimateontheporjection} is automatically satisfied.

\subsubsection{Decay Estimate on Line Bundles}
We now introduce a decomposition of the curvature. For the Hermitian bundle $(E,H)$ with Higgs bundle $(\bar{\pa}_E,\vp)$ and corresponding solution $(A,\phi)$. Let $s, s'$ be sections of $E$. We define a new approximate metric 
\begin{equation}
	\label{eq_approximate_metric}
	\begin{split}
		H_{\app}(s, s') := \sum_{i=1}^r H(\pi_i s, \pi_i s'),
	\end{split}
\end{equation}
 which is diagonal with respect to the eigenvalue decomposition. In particular, $H_{\app}|_{E_i} = H|_{E_i}$. Let $B_R$ be a ball centered at $x$.

We define another Chern connection for the metric $H_{\app}$, called the approximate Chern connection, which is given by:
\begin{equation}
	\label{eq_approximate_Chern}
	D_{\app} := \bar{\partial}_E + \vp + \sum_{i=1}^r (\partial_i+ \vp^{\dagger}_i) = \sum_{i=1}^r (\bar{\partial}_i+ \vp_i + \partial_i+ \vp^{\dagger}_i),
\end{equation}
where $\partial_i:= \pi_i' \circ \partial_E \circ \pi_i$ and $\vp^{\dagger}_i := \pi_i' \circ \vp^{\dagger} \circ \pi_i$.

The connection $\bar{\partial}_i+ \vp_i + \partial_i+ \vp^{\dagger}_i$ is a connection on $E_i$, which might not be flat but characterizes the asymptotic behavior of $D_H$. In particular, the difference between $D_H$ and $D_{\app}$ becomes small as the size of the Higgs field grows large:

\begin{proposition}
	\label{prop_approximate_Chern-estimate}
	There exist positive constants $C, C'$ such that:
	\[
	|D_H - D_{\app}| \leq C e^{-C' d_x}
	\]
	over $B_R$, where the constants depend on $\mbs_{\hat{\vp}}$ and the holomorphic structure $\bar{\partial}_E$.
\end{proposition}

\begin{proof}
	A straightforward computation yields:
	\[
	|D_H - D_{\app}| = |\partial_E - \sum_{i=1}^r \partial_i| + |\vp^{\dagger} - \sum_{i=1}^r \vp^{\dagger}_i|.
	\]
	We compute:
	\begin{equation*}
		\begin{split}
			\partial_E s &= \sum_{i=1}^r \pi_i \circ \partial_E \circ \pi_i s + \sum_{i \neq j} \pi_i \circ \partial_E \circ \pi_j s \\
			&= \sum_i \pi_i \circ \partial_E \circ \pi_i s + \sum_{i \neq j} \pi_i \circ \pi_j \circ \partial_E s + \sum_{i \neq j} \pi_i \circ (\partial_E \pi_j) \circ s.
		\end{split}
	\end{equation*}
	For the second and third terms, since $\pi_i \circ \pi_j = 0$ for $i \neq j$ and by Proposition \ref{estimateontheporjection}, $|\pi_i \circ (\partial_E \pi_j) \circ s| \leq C e^{-C' d_x} |s|$, we only need to compare the first term with $\sum_{i=1}^r \partial_i$.
	
	By Proposition \ref{estimateontheporjection} and since $\pi_i|_{E_i} = \pi_i'|_{E_i}$, we have:
	\begin{equation}
		\begin{split}
			|\pi_i \circ \partial_E \circ \pi_i s - \partial_is| &= |\pi_i \circ \partial_E \circ \pi_i s - \pi_i' \circ \partial_E \circ \pi_i s| \\
			&= |\pi_i' \circ (\partial_E \pi_i) \circ \pi_i s| \leq C e^{-C' d_x} |s|.
		\end{split}
	\end{equation}
	Therefore, we obtain $|\partial_E s - \sum_{i=1}^r \partial_is| \leq C e^{-C' d_x} |s|$.
	
	For $\vp^{\dagger} - \sum_{i=1}^r \vp^{\dagger}_i$, we write $\vp^{\dagger} = \sum_{i,j=1}^r \pi_i \circ \vp^{\dagger} \circ \pi_j$. For $i \neq j$, by Corollary \ref{corocommutevanishofpi}, we have $|\pi_i \circ \vp^{\dagger} \circ \pi_j| = |[\pi_i, \vp^{\dagger}] \circ \pi_j| \leq C e^{-C' d_x}$. For $i = j$, we compute:
	\[
	|\pi_i \circ \vp^{\dagger} \circ \pi_i - \vp^{\dagger}_i| = |\pi_i \circ \vp^{\dagger} \circ \pi_i - \pi_i' \circ \vp^{\dagger} \circ \pi_i| = |\chi_i \circ \vp^{\dagger} \circ \pi_i| \leq C e^{-C' d_x},
	\]
	where the last inequality follows from Proposition \ref{prop_estimateonchi} and Proposition \ref{Prop_higgsbundlesemisimpledecompositionestimate}. The claim follows.
\end{proof}

We now define $H_i := H_{\app}|_{E_i}$, the Chern connection on the line bundle $(E_i, \bar{\partial}_i, H_i)$, denoted $d_{A_i} := \bar{\partial}_i + \partial_i$. For any section $s$ of $E_i$, we have $\bar{\partial}_i s = \bar{\partial}_E s$ and $\partial_is = \pi_i' \circ \partial_A s$.

\begin{proposition}
	\label{prop_linebundle}
	For the approximate curvature $F_{A_i}$ on $E_i$, over $B_R$, we define $F_{A_i}^{\perp}:=F_{A_i}^{1,1}-\gamma(E)$, then we have:
	\[
	|\Lambda F_{A_i}^{\perp}| \leq C e^{-C' d_x}.
	\]
	Additionally, over $B_R$, we have:
	\begin{equation}
		\label{eq_curvature_L2_estimate_linebundle}
		\int_{B_R} |F_{A_i}|^2 \leq (1 + C e^{-C' d_x}) \int_{B_R} |F_A|^2 + C R^n e^{-C' d_x}.
	\end{equation}
	
	Moreover, if $\deg(E) = 0$ and $\ch_2(E) \cdot [\omega]^{n-2} = 0$, then $|F_{A_i}| \leq C e^{-C' d_x}$.
\end{proposition}

\begin{proof}
	Let $s$ be any section of $E_i$. We compute:
	\begin{equation}
		\begin{split}
			\bar{\partial}_i \partial_is &= \bar{\partial}_E \circ \pi_i' \circ \partial_E (\pi_i s) \\
			&= \bar{\partial}_E \circ \pi_i \circ \partial_E s + \bar{\partial}_E \circ (\pi_i' (\partial_E \pi_i)) s \\
			&= \pi_i \circ \bar{\partial}_E \partial_E s + (\bar{\partial}_E \pi_i') \circ (\partial_E \pi_i) s + \pi_i' \circ (\bar{\partial}_E \partial_E \pi_i) s - \pi_i' \circ (\partial_E \pi_i) \circ \bar{\partial}_E s,
		\end{split}
	\end{equation}
	where we use $\pi_i' \circ \pi_i = \pi_i$ and $\bar{\partial}_E \pi_i = 0$.
	
	Similarly, we have:
	\begin{equation}
		\begin{split}
			\partial_i\bar{\partial}_i s &= \pi_i' \partial_E \bar{\partial}_E s = \pi_i' \partial_E (\pi_i \bar{\partial}_E s) \\
			&= \pi_i\circ\partial_E \bar{\partial}_E s + \pi_i'\circ(\partial_E \pi_i)\circ\bar{\partial}_E s.
		\end{split}
	\end{equation}
	Therefore, we obtain:
	\[
	F_{A_i} s = (\bar{\partial}_i \partial_i+ \partial_i\bar{\partial}_i) s = \pi_i F_A s + (\bar{\partial}_E \pi_i') \circ (\partial_E \pi_i) s + \pi_i'\circ(\bar{\partial}_E \partial_E \pi_i) s.
	\]
	
	By Proposition \ref{estimateontheporjection}, we have $|(\bar{\partial}_E \pi_i') \circ (\partial_E \pi_i)| \leq C e^{-C' d_x}$. Since $\bar{\partial}_E \pi_i = 0$, we also have $\bar{\partial}_E \partial_E \pi_i = [F_A, \pi_i]=[F_A^{\perp},\pi_i]$. In addition, we have $\pi_i\Lam F_As=\pi_i\Lam F_A^{\perp} s+\gamma(E)s$. By Proposition \ref{Theoremcurvatureestimate}, $|\Lambda \bar{\partial}_E \partial_E \pi_i| \leq C e^{-C' d_x}$, which implies $|\Lambda F_{A_i}^{\perp}| \leq C e^{-C' d_x}$.
	
	By Lemma \ref{Lemma_factsabouttheseprojection}, since $|\pi_i'| = 1$ and $|\pi_i| \leq 1 + C e^{-C' d_x}$, we have:
	\[
	\int_{B_R} |F_{A_i}|^2 \leq (1 + C e^{-C' d_x}) \int_{B_R} |F_A|^2 + C R^n e^{-C' d}.
	\]
	
	Furthermore, if $\deg(E) = 0$ and $\ch_2(E) \cdot [\omega]^{n-2} = 0$, then by Proposition \ref{Theoremcurvatureestimate}, $|F_A| \leq C e^{-C' d_x}$, which implies $|F_{A_i}| \leq C e^{-C' d_x}$.
\end{proof}

\begin{subsection}{Higher Derivative Estimates}
	In this subsection, we discuss the higher derivative estimates for the connection and metrics, which generalizes slightly the construction in Mochizuki \cite[Section 2.2.5]{Mochizukiasymptotic}.
	
	\begin{lemma}
		\label{lem_construction_rankone}
		For the rank-one Higgs bundle $(E_i, \bpa_i, \vp_i)$ with induced Hermitian metric $h_i$ and connection $d_{A_i} = \pa_i + \bpa_i$, there exists a constant $\epsilon_0$ such that, over $B_R$, if 
		\[\int_{B_R} |F_{A_i}|^2 \leq \epsilon_0 \quad \text{and} \quad \| \Lambda F_{A_i} \|_{L^{\infty}} < \epsilon_0,\]
		then there exists a holomorphic frame $u_i$ (i.e., $\bpa_i u_i = 0$) such that
		\begin{equation}
			|u_i|_x = 1, \quad |\pa_i u_i| \leq C_{\epsilon_0},
		\end{equation}
		where $C_{\epsilon_0}$ is a constant that depends only on $\epsilon_0$.
		
		Moreover, if $\deg(E) = \ch_2(E) \cdot \omega^{n-2} = 0$, then 
		\begin{equation}
			|u_i|_x = 1, \quad |\pa_i u_i| \leq C e^{-C'd_x},
		\end{equation}
		for some constants $C$ and $C'$.
	\end{lemma}
	
	\begin{proof}
		By Theorem \ref{Uhlenbeckcompactnessforlocalmodel}, in the Uhlenbeck gauge, we have a trivialization $E_i \cong B_R \times \mathbb{C}$, and we can write $d_{A_i} = d + A_i$, where $A_i$ is a 1-form on $B_R$ that satisfies $\|A_i\|_{L_1^p} \leq C_{\epsilon_0}$ for any $p$. We write $d = \pa + \bpa$ and let $s_i$ be the constant frame on $B_R$, with $d s = 0$, so that $(\bpa_i + \pa_i)s_i = A_i s_i$.
		
		Since $E_i$ is a line bundle and $A_i$ is an integrable connection, we have $\bpa A_i^{0,1} = 0$. Therefore, there exists a function $\rho_i$ such that $\bpa \rho_i = A_i^{0,1}$. By the  construction of $\bar{\pa}$-Poincare Lemma, we have $\|\rho_i\|_{L_1^p} \leq C \|A_i^{0,1}\|_{L_1^p} \leq C_{\epsilon_0}$. Moreover, since $\bpa^* \rho_i = 0$, by the elliptic estimate for the operator $\bpa + \bpa^*$, we obtain $\|\rho_i\|_{L_2^p} \leq C_{\epsilon_0}$.
		
		We define the frame $u_i := e^{-\rho_i} s_i$, which satisfies $C_{\epsilon_0}^{-1} \leq |u_i| \leq C_{\epsilon_0}$. By construction, we have $\bpa_i u_i = 0$, and
		\begin{equation}
			\label{eq_linebundle_term_estimate}
			|\pa_i u_i| = |(-\pa \rho_i + A_i^{1,0}) u_i| \leq C_{\epsilon_0}.
		\end{equation}
		After normalizing the value of $u_i$ at $x$, we obtain the desired frame.
		
		When $\deg(E) = \ch_2(E) \cdot \omega^{n-2} = 0$, by Proposition \ref{prop_linebundle}, we have $|F_{A_i}| \leq C e^{-C'd_x}$, implying $\|A_i\|_{L_1^p} \leq C e^{-C'd_x}$. The proof then follows in the same way, with the constants $C_{\epsilon_0}$ replaced by $C e^{-C'd_x}$.
	\end{proof}
	
	Now, we obtain a holomorphic frame $\mathbf{u} = (u_1, \dots, u_r)$ that trivializes $E = \oplus_{i=1}^r E_i$ on $B_R$. We define an $r \times r$ Hermitian matrix $G_{ij} := H(u_i, u_j)$. By the construction of $u_i$ in Lemma \ref{lem_construction_rankone}, we have $G_{ii}|_x = 1$, $|G_{ii}| + |\log G_{ii}| \leq C_{\epsilon_0}$. By Corollary \ref{corocommutevanishofpi}, we have $|H(u_i, u_j)| \leq C e^{-C'd_x}$ for $i\neq j$, which implies $|(G^{-1})_{ii}| \leq C_{\epsilon_0}$. In addition, if $\deg(E) = \ch_2(E) \cdot \omega^{n-2} = 0$, then $|\log G_{ii}| \leq C e^{-C'd_x}$.
	
	We define $\Theta$ to be the holomorphic $\Omega_X^{1,0}$-valued $r \times r$ matrix defined by $\varphi \mathbf{u} = \mathbf{u} \cdot\Theta$. Since $\varphi = \oplus \lambda_i \mathrm{Id}_{E_i}$, we have $\Theta_{ij} = 0$ for $i \neq j$ and $\Theta_{ii} = \lambda_i$. For the trivial connection $d$, we write $d = \pa + \bpa$.
	
	\begin{lemma}
		\label{lem_estimate_for_G}
		Over $B_R$, for $i, j = 1, \dots, r$ with $i \neq j$, the Hermitian metric $G$ satisfies
		\begin{equation}
			\begin{split}
				|G_{ij}| &\leq C e^{-C'd_x}, \quad |(G^{-1})_{ij}| \leq C e^{-C'd_x}, \\
				|\pa G_{ij}| = |\bpa G_{ij}| &\leq C e^{-C'd_x}, \quad |\pa G_{ii}| = |\bpa G_{ii}| \leq C_{\epsilon_0}.
			\end{split}
		\end{equation}
		In addition, if $\deg(E) = 0$ and $\ch_2(E) \cdot [\omega]^{n-2} = 0$, then $|\pa G_{ii}| = |\bpa G_{ii}| \leq C e^{-C'd_x}$.
	\end{lemma}
	
	\begin{proof}
		By Corollary \ref{corocommutevanishofpi}, we have $|H(u_i, u_j)| \leq C e^{-C'd_x}$, which implies the estimates for $G_{ij}$ and $(G^{-1})_{ij}$. We compute
		\[
		\pa H(u_i, u_j) = H(\bpa_E u_i, u_j) + H(u_i, \pa_E u_j) = H(u_i, \pa_E u_j) = H(u_i, (\pa_E \pi_j) u_j) + H(u_i, \pi_j \pa_E u_j).
		\]
		By Proposition \ref{estimateontheporjection}, we have $|\pa_E \pi_j| \leq C e^{-C'd_x}$. In addition, for any section $s \in E_j$, we have
		\[
		|\pi_j \circ \pa_E s - \pa_j s| = |\pi_j' \circ \pi_j \circ \pa_E s - \pi_j' \circ \pa_E \circ \pi_j s| = |\pi_j' \circ (\pa_E \pi_j) s| \leq C e^{-C'd_x} |s|.
		\]
		Furthermore, by \eqref{eq_linebundle_term_estimate}, we can write $\pa_j u_j = (-\pa \rho_j + A_j^{1,0}) u_j$, with $| -\pa \rho_j + A_j^{1,0}| \leq C_{\epsilon_0}$. Therefore,
		\[
		|H(u_i, \pa_j u_j)| \leq C_{\epsilon_0} |H(u_i, u_j)| \leq C e^{-C'd_x},
		\]
		since $|H(u_i, u_j)| \leq C e^{-C'd_x}$. Similarly, we have
		\[
		|H(u_i, \pa_i u_i)| \leq C_{\epsilon_0} |H(u_i, u_i)| \leq C_{\epsilon_0}.
		\]
		If $\deg(E) = 0$ and $\ch_2(E) \cdot [\omega]^{n-2} = 0$, then $| -\pa \rho_j + A_j^{1,0}| \leq C e^{-C'd_x}$, which implies $|\pa G| = |\bpa G| \leq C e^{-C'd_x}$ by straightforward computation.
	\end{proof}
	
	Next, we provide higher derivative estimates for $G$ and $\Theta$. Let $\nabla$ denote the covariant derivative on $E$ defined by the trivial connection. Using $\nabla$, we define the Sobolev spaces
	\[L_k^p := \left\{ f \mid \sum_{i=1}^k \left( \int_{B_R} |\nabla^i f|^p \right)^{1/p} < \infty \right\}.
	\]
	Then, in the frame $\mathbf{u}$, the Hitchin-Simpson equations \eqref{Hitchin-Simpson} for $G$ and $\Theta$ can be written as
	\begin{equation}
		\label{eq_Hitchin_simpson_matrix_form}
		\Lambda(\bpa(G^{-1} \pa G) + [\Theta, \Theta^{\dagger}]) = 0, \quad \bpa \Theta = 0, \quad \Theta \wedge \Theta = 0,
	\end{equation}
	where $\Theta^{\dagger} := G^{-1} \bar{\Theta}^{\intercal} G$.
	
	Since $|\Theta| \leq C d$, from the equation $\bpa \Theta = 0$, by elliptic estimates, we have
	\begin{equation}
		\label{eq_estimate_higherderivates_Higgsfield}
		|\nabla^k \Theta| \leq C d,
	\end{equation}
	for any non-negative integer $k$.
	
	The situation we consider is the asymptotic behavior as $d$ becomes sufficiently large. To estimate $G$, we need to show that the contribution from the quadratic term $[\Theta, \Theta^{\dagger}]$ is independent of $d$. In addition, the asymptotic behavior of the diagonal and off-diagonal components of $G$ will differ, so we rewrite equation \eqref{eq_Hitchin_simpson_matrix_form} for the diagonal and off-diagonal components of $G$.
	
	To simplify notation, define
	\[P := (G^{-1} \bpa G)(G^{-1} \pa G), \quad Q := [\Theta, \Theta^{\dagger}],\]
	and $R := - (\Lambda P) - Q$. Then $G$ satisfies the equation
	\[\Lambda \bpa \pa G = G R.
	\]
	
	For each component of $G$, for $i \neq j$, we have
	\begin{equation}
		\label{eq_different_component_estimates}
		\begin{split}
			\Lambda \bpa \pa G_{ii} &= G_{ii} R_{ii} + \sum_{k \neq i} G_{ik} R_{ki}, \\
			\Lambda \bpa \pa G_{ij} &= G_{ii} R_{ij} + G_{ij} R_{jj} + \sum_{k \neq i, j} G_{ik} R_{kj}.
		\end{split}
	\end{equation}
	Here $(\cdot)_{ij}$ denotes the $(i, j)$-th component of the matrix $(\cdot)$.
	
	We make the following simple but important observation:
	
	\begin{lemma}
		\label{lem_matrix_diagonal}
		Let $M$ and $N$ be $r \times r$ matrices, and suppose $M$ is diagonal. Then $([M, N])_{ii} = 0$, and $|[M, N]| \leq 2|M| \sum_{i \neq j} |N_{ij}|$.
	\end{lemma}
	
	\begin{proof}
		The proof follows from straightforward computation.
	\end{proof}
	
	We have the following bootstrapping argument for the higher derivatives of $G$.
	
	\begin{proposition}
		\label{prop_higher_derivative_estimate_forG}
		Pver $B_R$, for any $k, p \geq 0$ and $1 \leq i \neq j \leq r$, we have
		\begin{equation}
			\label{eq_higher_derivatives_G}
			\begin{split}
				\|G_{ij}\|_{L^p_{k+1}} &\leq C e^{-C'd_x}, \quad \|G_{ii}\|_{L^p_{k+1}} \leq C_{\epsilon_0}.
			\end{split}
		\end{equation}
		In addition, suppose $\deg(E) = 0$ and $\ch_2(E) \cdot [\omega]^{n-2} = 0$. Then $|\nabla^{k+1} G| \leq C e^{-C'd_x}$.
	\end{proposition}
	
	\begin{proof}
		When $k = 0$, the estimates for any $p$ are given by Lemma \ref{lem_estimate_for_G}. Suppose the estimate \eqref{eq_higher_derivatives_G} holds for $k$. Since $\Lambda \bpa \pa$ is an elliptic operator, it suffices to estimate $\|\nabla^k \Lambda \bpa \pa G_{ii}\|_{L^p}$ and $\|\nabla^k \Lambda \bpa \pa G_{ij}\|_{L^p}$ from \eqref{eq_different_component_estimates}. The proof proceeds in the following steps.
		
		For any $p$, we first prove
		\begin{equation}
			\label{eq_estimateforP}
			\|\nabla^k P_{ii}\|_{L^p} \leq C_{\epsilon_0}, \quad \|\nabla^k P_{ij}\|_{L^p} \leq C e^{-C'd_x}.
		\end{equation}
		
		For positive integers $k_1, k_2$ with $k = k_1 + k_2$, we can write 
		\[\nabla^k P_{ii} = \sum_{k_1 + k_2 = k} \nabla^{k_1} (G^{-1} \bpa G) \otimes \nabla^{k_2} (G^{-1} \pa G).
		\]
		For any $p$, as $k_1 \leq k$ and $\|G\|_{L^{p}_{k+1}} \leq C_{\epsilon_0}$, the term $\nabla^{k_1} G^{-1} \bpa G$ involves at most $k+1$ derivatives of $G$, and by H\"older's inequality, $\|\nabla^{k_1} G^{-1} \bpa G\|_{L^p} \leq C_{\epsilon_0}$. Similarly, $\|\nabla^{k_2} (G^{-1} \pa G)\|_{L^p} \leq C_{\epsilon_0}$. Therefore, $\|\nabla^k P_{ii}\|_{L^p} \leq C_{\epsilon_0}$.
		
		For $P_{ij}$, we first write $P_{ij} = \sum_q (G^{-1} \bpa G)_{iq} (G^{-1} \pa G)_{qj}$. Then,
		\[\nabla^k P_{ij} = \sum_q \sum_{k_1 + k_2 = k} \nabla^{k_1} (G^{-1} \bpa G)_{iq} \cdot \nabla^{k_2} (G^{-1} \pa G)_{qj}.
		\]
		Since $i \neq j$, and the off-diagonal components of $G$ decay exponentially, for $q \neq i$ or $q \neq j$, and for any $p$, we have $\|\nabla^{k_1} (G^{-1} \bpa G)_{iq}\|_{L^p} \leq C e^{-C'd_x}$ and $\|\nabla^{k_2} (G^{-1} \pa G)_{qj}\|_{L^p} \leq C e^{-C'd_x}$. In addition, $\|\nabla^{k_1} (G^{-1} \bpa G)_{ii}\|_{L^p} \leq C_{\epsilon_0}$ and $\|\nabla^{k_2} (G^{-1} \pa G)_{jj}\|_{L^p} \leq C_{\epsilon_0}$. Thus, $\|\nabla^k P_{ij}\|_{L^p} \leq C e^{-C'd_x}$.
		
		Next, we prove
		\begin{equation}
			\label{eq_estimateforQ}
			Q_{ii} = 0, \quad \|\nabla^k Q_{ij}\|_{L^p} \leq C e^{-C'd_x}.
		\end{equation}
		Since $\Theta$ is diagonal, by Lemma \ref{lem_matrix_diagonal}, $Q_{ii} = 0$.
		
		We write $\nabla^k Q = \sum_{k_1 + k_2 = k} [\nabla^{k_1} \Theta, \nabla^{k_2} \Theta^{\dagger}]$. Since $\nabla^{k_1} \Theta$ is still diagonal, by Lemma \ref{lem_matrix_diagonal},
		\[|[\nabla^{k_1} \Theta, \nabla^{k_2} \Theta^{\dagger}]| \leq |\nabla^{k_1} \Theta| \cdot \sum_{i \neq j} |(\nabla^{k_2} \Theta^{\dagger})_{ij}|.
		\]
		By \eqref{eq_estimate_higherderivates_Higgsfield}, $|\nabla^{k_1} \Theta| \leq C d$. Recall that $\Theta^{\dagger} = G^{-1} \bar{\Theta}^{\intercal} G$ and that $\bar{\Theta}^{\intercal}$ is diagonal, so the off-diagonal term can be expressed as
		\[(\nabla^{k_2} \Theta^{\dagger})_{ij} = \sum_{q=1}^r \sum_{l_1 + l_2 + l_3 = k_2} \nabla^{l_1} (G^{-1})_{iq} \cdot (\nabla^{l_2} \bar{\Theta}^{\intercal}_{qq}) \cdot \nabla^{l_3} G_{qj}.
		\]
		Since $i \neq j$, and the derivatives of the off-diagonal terms of $G$ are bounded by $C e^{-C'd_x}$, by \eqref{eq_estimate_higherderivates_Higgsfield}, we obtain $|(\nabla^{k_2} \Theta^{\dagger})_{ij}| \leq C d e^{-C'd_x}$. Therefore, $\|\nabla^k Q_{ij}\|_{L^p} \leq C e^{-C'd_x}$.
		
		By equations \eqref{eq_estimateforP} and \eqref{eq_estimateforQ}, we conclude that
		\[\|\nabla^k \Lambda \bpa \pa G_{ii}\|_{L^p} \leq C_{\epsilon_0}, \quad \|\nabla^k \Lambda \bpa \pa G_{ij}\|_{L^p} \leq C e^{-C'd_x}.
		\]
		By elliptic regularity, we obtain $\|G_{ij}\|_{L^p_{k+2}} \leq C e^{-C'd_x}$ and $\|G_{ii}\|_{L^p_{k+2}} \leq C_{\epsilon_0}$, completing the induction.
		
		When $\deg(E) = 0$ and $\ch_2(E) \cdot [\omega]^{n-2} = 0$, every estimate with bound $C_{\epsilon_0}$ can be replaced by $C e^{-C'd_x}$. Therefore, we conclude $|\nabla^{k+1} G| \leq C e^{-C'd_x}$ for $k \geq 0$.
	\end{proof}
	
	Since $|\nabla_A^k F_A|$ can be expressed in terms of derivatives of $G$, we obtain the following higher derivative estimates for the curvature. By \eqref{eq_curvature_L2_estimate_linebundle}, when $d_x$ is sufficiently large, the assumptions in Lemma \ref{lem_construction_rankone}, Lemma \ref{lem_estimate_for_G}, and Proposition \ref{prop_higher_derivative_estimate_forG} are satisfied, and we conclude the following higher derivative estimates:
	
	\begin{corollary}
		\label{cor_higher_derivatives_curvature}
		For each $k$, we have
		\[|\nabla_A^k F_A| \leq C_{\epsilon_0}.
		\]
		Suppose $\deg(E) = 0$ and $\ch_2(E) \cdot [\omega]^{n-2} = 0$. Then 
		\[|\nabla_A^k F_A| \leq C e^{-C'd_x}.
		\]
	\end{corollary}
\end{subsection}
	\end{section}
	
\begin{section}{Compactness for Sequences of Solutions}
	In this section, we prove a compactness theorem for sequences of solutions to the Hitchin-Simpson equations. Compactness theorems for similar equations have been proved by Taubes \cite{taubes2013compactness} over real 3- and 4-manifolds and by Mochizuki for Riemann surfaces \cite{Mochizukiasymptotic}.
	
	\begin{subsection}{Uhlenbeck Compactness}
		We first introduce the Uhlenbeck compactness theorem for integrable connections with suitable bounds. Let $R_0$ be the injectivity radius of $X$. Using the Uhlenbeck patching argument \cite{donaldsonkronheimer1986geometry}, we have the following compactness statement:
		
		\begin{theorem}[\cite{uhlenbeck1982connections,uhlenbeck1986aprior}]
			\label{Uhlenbeckcompactness}
			Let $U$ be a K\"ahler manifold (not necessarily compact) and let $E \to U$ be a Hermitian vector bundle with metric $H$. For any $p > n$, let $A_j$ be a sequence of integrable unitary connections on $E \to U$ such that $\|F_{A_j}\|_{L^2(U)}$ and $\|\Lambda F_{A_j}\|_{L^{\infty}(U)}$ are uniformly bounded. Then, after passing to a subsequence, there exist:
			\begin{itemize}
				\item[(i)] a closed subset $Z_{\Uh} \subset U$ with Hausdorff codimension at least 4,
				\item[(ii)] a smooth Hermitian vector bundle $(E_{\infty}, H_{\infty})$ defined on $U \setminus Z_{\Uh}$, which can be identified with $(E, H)$ over $U \setminus Z_{\Uh}$,
				\item[(iii)] for any compact subset $W \subset U \setminus Z_{\Uh}$, there exist $L^p_2$ isometries $g_i: (E_{\infty}, H_{\infty})|_W \to (E, H)|_W$ such that $g_i(A_i)$ converges weakly to $A_{\infty}$ in the $L_1^p$ topology, where $A_{\infty}$ satisfies $\int_U |F_{A_{\infty}}|^2 < \infty$,
				\item[(iv)] the singular set $Z_{\Uh}$ can be described as
				\begin{equation}
					\label{eq_singularity_Uhlenbeck}
					Z_{\Uh} =\bigcap_{0<R<R_0} \left\{ x \in U : \liminf_{i \to \infty} R^{4 - 2n} \int_{B_R(x)} |F_{A_i}|^2 dV \geq \epsilon_0 \right\},
				\end{equation}
				where $\epsilon_0$ depends only on the geometry of $U$ and the bounds on $\|F_{A_i}\|_{L^2(U)}$ and $\|\Lambda F_{A_i}\|_{L^{\infty}(U)}$.
			\end{itemize}
		\end{theorem}
		
		For the limiting connection $A_{\infty}$, we have the following theorem, which is a combination of results from \cite{BandoSiu1994,daskalopoulos2004convergence}; see also \cite{sibley2015asymptotics}.
		\begin{proposition}
			\label{Thm_curvatureconvergencewithderivativeestimate}
			Under the previous assumptions, suppose $\|d_{A_j} \Lambda F_{A_j}\|_{L^2} \to 0$. Then $(A_{\infty}, \phi_{\infty})$ is a Hermitian-Yang-Mills connection. Moreover, $\Lambda F_{A_j} \to \Lambda F_{A_{\infty}}$ in $L^p_{\text{loc}}(U \setminus Z_{\Uh})$. Additionally, over $U \setminus Z_{\Uh}$, there is a holomorphic, orthogonal splitting
			\[(E_{\infty}, H_{\infty}, \nabla_{A_{\infty}}) = \bigoplus_{i=1}^l (E_{\infty,i}, h_{\infty,i}, \nabla_{A_{\infty},i}).\]
		\end{proposition}
	\end{subsection}

\begin{subsection}{Compactness Theorem for Bounded Sequences}
	We now state our convergence theorem for solutions with bounded $L^2$ norms of Higgs fields. The following result is well-known and follows by applying Uhlenbeck's compactness theorem and standard elliptic theory.
	
	\begin{theorem}
		\label{Thm_Uhlenbeckbubbleboundedsequence}
		Let $(E, H)$ be a Hermitian bundle, and let $(A_i, \phi_i = \varphi_i + \varphi_i^{\dagger})$ be a sequence of solutions to the Hitchin-Simpson equations \eqref{Hitchin-Simpson}. Suppose that $\|\phi_i\|_{L^2(X)} \leq C$. Then there exist a closed singular set $Z_{\Uh}$ of codimension 4, defined by
		\begin{equation}
			Z_{\Uh} := \bigcap_{0<R<R_0} \left\{ x \in X : \liminf_{i \to \infty} R^{4 - 2n} \int_{B_R(x)} |F_{A_i}|^2 dV \geq \epsilon_0 \right\},
		\end{equation}
		with finite Hausdorff measure, a smooth vector bundle $(E_{\infty}, H_{\infty})$, and a pair $(A_{\infty}, \phi_{\infty})$ such that the following holds:
		\begin{itemize}
			\item[(i)] After passing to a subsequence, over every compact subset $W \subset X \setminus Z_{\Uh}$, there exist unitary gauge transformations $g_i: (E, H)|_W \to (E_{\infty}, H_{\infty})|_W$ such that $g_i(A_i, \phi_i)$ converges smoothly to $(A_{\infty}, \phi_{\infty})$.
			\item[(ii)] $(E_{\infty}, H_{\infty}, A_{\infty}, \phi_{\infty})$ satisfies the Hitchin-Simpson equations over $X \setminus Z_{\Uh}$.
		\end{itemize}
	\end{theorem}
	
	\begin{proof}
		By Proposition \ref{eigenvalueestimate}, we obtain that $|\varphi_i|_h \leq C$. Since $A_i, \phi_i$ satisfy the Hitchin-Simpson equation $i \Lambda (F_{A_i} + [\varphi_i, \varphi_i^{\dagger}]) = 0$, we have $|\Lambda F_{A_i}|_{\MC^0} \leq C$. Applying Uhlenbeck's compactness theorem \ref{Uhlenbeckcompactness}, we obtain a closed set $Z_{\Uh}$ such that $g_i^* A_i \rightharpoonup A_{\infty}$ in $L^p_{1, \text{loc}}$ over $X \setminus Z_{\Uh}$. For simplicity, we omit the unitary gauge $g_i$ and still write $(A_i, \phi_i)$ to represent $(g_i(A_i), g_i(\phi_i))$. We define $a_i := \bar{\partial}_{A_i} - \bar{\partial}_{A_{\infty}}$, then $\bar{\partial}_{A_{\infty}}(\varphi_i) = \bar{\partial}_{A_i} \varphi_i + [a_i, \varphi_i]$. Since $a_i \rightharpoonup 0$ in $L^p_{1, \text{loc}}$, we have strong convergence in $L^p_{\text{loc}}$. In addition, since $\varphi_i \in C^0$, we obtain $\varphi_i \in L^p_{1, \text{loc}}$. As $H_{\infty}$ is a fixed smooth metric, we also have $\varphi_i^{\dagger} \in L^p_{1, \text{loc}}$. Thus, there exist $\varphi_{\infty}, \varphi_{\infty}^{\dagger}$ such that $\phi_i = \varphi_i + \varphi_i^{\dagger}$ converges to $\phi_{\infty} := \varphi_{\infty} + \varphi_{\infty}^{\dagger}$. By a standard bootstrapping argument for elliptic equations, we obtain smooth convergence.
		
		Furthermore, since $(E, H, A_i, \phi_i)$ satisfies \eqref{Hitchin-Simpson} and the convergence is in $L^p_{1, \text{loc}}$, we conclude that $(E_{\infty}, H_{\infty}, A_{\infty}, \phi_{\infty})$ also satisfies the Hitchin-Simpson equations over $X \setminus Z_{\Uh}$.
	\end{proof}
	
	The following description of the limiting solution $(E_{\infty}, H_{\infty}, A_{\infty}, \phi_{\infty})$ is due to Chen \cite{chen2024vafawitten}, where the bubbling set itself has a complex analytic structure. The proof is identical, so we will only state the result.
	
	\begin{proposition}[\cite{chen2024vafawitten}, Proposition 5.6]
		The set $Z_{\Uh}$ is a real subvariety of $X$ with codimension at least four. In addition, $A_{\infty}$ defines a unique reflexive sheaf $\mathcal{E}_{\infty}$ over $X$, and $\phi_{\infty}$ extends to a global section of $\Omega_X^{1,0}(\End(\mathcal{E}_{\infty}))$. Moreover, $Z_{\Uh}$ admits a decomposition
		\[
		Z_{\Uh} = \bigcup_k Z_k \cup \mathrm{Sing}(\mathcal{E}_{\infty}),
		\]
		where $Z_k$ are the irreducible pure codimension 4 components of $Z_{\Uh}$. As a sequence of currents, we have
		\[
		\operatorname{Tr}(F_{A_i} \wedge F_{A_i}) \to \operatorname{Tr}(F_{A_{\infty}} \wedge F_{A_{\infty}}) + 8\pi^2 \sum_k m_k Z_k,
		\]
		where $m_k$ are non-negative integers.
	\end{proposition}
	
\end{subsection}
\begin{subsection}{Estimate for Sequences of Solutions}
	Let $(E,H)$ be a Hermitian vector bundle, let $(A_i,\phi_i)$ be a sequence of solutions to \eqref{Hitchin-Simpson}, where $r_i:=\|\phi_i\|_{L^2(X)}$ and $\hat{\phi}_i:=r_i^{-1}\phi_i=\hvp_i+\hvp_i^{\dagger}$. We denote $\mbs_i=(p_1(\hvp_i),\cdots,p_r(\hvp_i))$, with the discriminant locus $Z_i$. By Theorem \ref{Theorem_convergencespectralcover}, there exists $\mbs_{\infty}$ such that $\mbs_i$ converges smoothly to $\mbs_{\infty}$ in $\MB_X$, and $Z_i$ converges to $Z_{\infty}$, where $Z_{\infty}$ is the discriminant locus of $\mbs_{\infty}$. 
	
	For each $x\in X\setminus Z_{\infty}$, for sufficiently large $i$, we have $x\notin Z_i$. Let $\lambda_1^i, \dots, \lambda_r^i$ be the $r$ distinct eigenvalues of $\phi_i$, and define $d_{x,i}:=\min_{k\neq l}|\lambda_k^i(x)-\lambda_l^i(x)|$ as the minimal difference between eigenvalues at $x$. Let $M_{x,i}=\max_{1\leq k\leq r}|\lambda_k^i(x)|$, and define the radius $R_{x,i}:=\min\{R_0, \frac{d_{x,i}}{100C\|\phi_i\|_{L^2(X)}}\}$, as given in Proposition \ref{Prop_higgsbundlesemisimpledecompositionestimate}. Then the following holds:
	
	\begin{proposition}
		\label{Prop_uniformbound}
		For each $x\in X\setminus Z_{\infty}$, there exist constants $i_0, R_{x,\infty}, d_{x,\infty}$ such that for all $i\geq i_0$, we have $R_{x,i}\geq R_x$ and $r_i^{-1}d_{x,i}\geq d_{x,\infty}$.
	\end{proposition}
	
	\begin{proof}
		Let $d_{x,\infty}$ be the minimal difference between the eigenvalues of $\phi_{\infty}$ at $x$. As $\mbs_i$ converges to $\mbs_{\infty}$, we have $\lim_{i\to\infty}r_i^{-1}d_{x,i}=d_{x,\infty}$. Since $R_{x,i}:=\min\{R_0, \frac{d_{x,i}}{100C\|\phi_i\|_{L^2(X)}}\}$, the bound for $R_{x,i}$ follows immediately.
	\end{proof}
	
	\begin{proposition}
		\label{Thm_citeconvergence}
		For the sequence $(A_i,\phi_i)$, let $x\in X\setminus Z_{\infty}$. There exist constants $i_0, d_{x,\infty}, R_{x,\infty}$, which depend only on $\mbs_{\infty}$, such that for all $i \geq i_0$ and over the ball $B_{R_{x,\infty}}$, the following estimates hold:
		\begin{equation*}
			\begin{split}
				&|\sqrt{-1}\Lambda F_{A_i}|\leq Ce^{-C'r_i d_{x,\infty}},\quad |[\hvp_i,\hvp_i^{\dagger}]|\leq Ce^{-C'r_i d_{x,\infty}},\\
				&\|d_{A_i}^{\ast}F_{A_i}\|_{L^p}\leq Cr_i e^{-C'r_i d_{x,\infty}}\|\nabla_{A_i}\hvp_i\|_{L^p}.
			\end{split}
		\end{equation*}
	\end{proposition}
	
	\begin{proof}
		This is a combination of Proposition \ref{Theoremcurvatureestimate}, Proposition \ref{Prop_curvaturederivativeestimates}, and Proposition \ref{Prop_uniformbound}.
	\end{proof}
	
	Define $Z_{\Uh}$ to be the Uhlenbeck singularity set determined by the connections $A_i$. For $x\in X\setminus (Z_{\infty}\cup Z_{\Uh})$, by Corollary \ref{cor_higher_derivatives_curvature}, we obtain the following curvature estimates:
	
	\begin{proposition}
		\label{prop_higher_derivative_curvature}
		There exist constants $i_0, d_{x,\infty}, R_{x,\infty}$, depending only on $\mbs_{\infty}$, such that over $B_{R_{x,\infty}}$, for each $k$, we have
		\begin{equation}
			|\nabla_{A_i}^k F_{A_i}|\leq C.
		\end{equation}
		Moreover, if $\deg(E_i)=0$ and $\ch_2(E_i).[\omega]^{n-2}=0$, then 
		\begin{equation}
			|\nabla_{A_i}^k F_{A_i}|\leq Ce^{-C' d_{x,\infty} r_i}.
		\end{equation}
	\end{proposition}
\end{subsection}
	\subsection{Compactness for Unbounded Sequences}
	We study the behavior of solutions to the Hitchin-Simpson equations with unbounded $L^2$ norms for the Higgs fields. Let $(A_i,\phi_i=\vp_i+\vp_i^{\dagger})$ be a solution to \eqref{Hitchin-Simpson}, with $r_i := \|\phi_i\|_{L^2(X)}$ and $\lim r_i = \infty$. 
	
	\begin{lemma}
		We have the following estimates:
		\label{Lemma_localL2controlofthecurvature}
		\begin{itemize}
			\item [(i)] $\lim_{i \to \infty} \int_X |\hat{\phi}_i \wedge \hat{\phi}_i|^2 = 0$, and $\lim_{i \to \infty} \int_X |d_{A_i}\hat{\phi}_i|^2 = 0$.
			\item [(ii)] For any subset $B \subset X$, we have $\int_B |F_{A_i}|^2 \leq \int_B |[\vp_i, \vp_i^{\dagger}]|^2 + C$.
		\end{itemize}
	\end{lemma}
	
	\begin{proof}
		By integrating \eqref{Eq_integralbyparts}, we obtain the bound $r_i^2 \int_X |\hat{\phi}_i \wedge \hat{\phi}_i|^2 \leq C$. In addition, using Proposition \ref{energyidentity}, we have the estimate $\int_X |F_{A_i} + [\vp_i, \vp_i^{\dagger}]|^2 + |d_{A_i} \phi_i|^2 \leq C$, and the result follows immediately.
	\end{proof}
	
	Next, we define the following adiabatic Hitchin-Simpson equations, which have been studied on Riemann surfaces in \cite{mazzeo2016ends, mazzeo2012limiting, fredrickson2018generic}. Over the Hermitian vector bundle $(E, H)$, we say that $(A, \phi = \vp + \vp^{\dagger})$ satisfies the \emph{adiabatic Hitchin-Simpson equations} if:
	\begin{equation}
		\label{adabatic-Hitchin-Simpson}
		\begin{split}
			&F_A^{(0,2)} = 0,\quad \vp \wedge \vp = 0,\quad \bar{\pa}_A \vp = 0,\quad \pa_A \vp = 0, \\
			&\sqrt{-1}\Lambda F_A^{\perp} = 0,\quad [\vp, \vp^{\dagger}] = 0,\quad \|\vp\|_{L^2(X)} = 1,
		\end{split}
	\end{equation}
	where $F_A^{\perp} := F_A^{1,1} - \gamma(E)\Id$. 
	
	\subsubsection{Main result} We now state the main theorem of this subsection.
	
	\begin{theorem}
		\label{Thm_mainconvergenceunboundedsequences}
		Let $(E, H)$ be a Hermitian vector bundle over $X$, and let $(A_i, \phi_i = \vp_i + \vp_i^{\dagger})$ be a sequence of solutions to \eqref{Hitchin-Simpson}, with $r_i := \|\phi_i\|_{L^2(X)}$ and $\hat{\phi}_i = r_i^{-1} \phi_i = \hat{\vp}_i + \hat{\vp}_i^{\dagger}$.
		
		Let $\MA_X$ be the Hitchin base, and $\kappa : \MMH \to \MB_X$ the Hitchin morphism and $\mbs_i:=\kappa(\hvp_i)\in \MB_X$, then
		\begin{enumerate}[label=(\roman*)]
			\item $\mbs_i$ converges smoothly to $\mbs_{\infty}\in \MB_X$.
			\item Suppose the discriminant locus $\Zta$ of the spectral cover defined by $\kappa_{\infty}$ satisfies $\Zta \subsetneq X$. Then there exists a Hausdorff codimension 4 set $Z_{\Uh} \subset X \setminus \Zta$, and a Hermitian bundle $(E_{\infty}, H_{\infty})$ over $X \setminus (Z_{\Uh} \cup \Zta)$ such that:
			\begin{enumerate}
				\item There exists a pair $(A_{\infty}, \phi_{\infty})$ over $(E_{\infty}, H_{\infty})$ such that, after passing to a subsequence and up to unitary gauge transformations, $(A_i, \phi_i)$ converges to $(A_{\infty}, \phi_{\infty})$ in the $\MC^{\infty}_{\loc}$ topology.
				\item $(A_{\infty}, \phi_{\infty})$ satisfies the adiabatic Hitchin-Simpson equations \eqref{adabatic-Hitchin-Simpson}.
				\item If $\deg(E) = 0$ and $\ch_2(E)[\omega]^{n-2} = 0$, then for any open set $U \subset X \setminus (Z_{\Uh} \cup \Zta)$, there exists a constant $d_U>0$ depends on $\mbs_{\infty}$ and $U$ such that 
				\begin{equation}
					\|(A_i, \phi_i) - (A_{\infty}, \phi_{\infty})\|_{\MC^k} \leq Ce^{-C' d_U r_i}.
				\end{equation}
			\end{enumerate}
		\end{enumerate}
	\end{theorem}
	
	\begin{proof}
		\textbf{Step 1: Convergence of the spectral cover.}
		
		We begin by proving the first part. Recall that for a given Higgs bundle $(\ME_i, \vp_i)$, the Hitchin morphism is $\mbs_i := \kappa(\hat{\vp}_i) = (p_1(\hat{\vp}_i), \cdots, p_r(\hat{\vp}_i)) \in \oplus_{k=1}^r H^0(\Sym^k \Omega_X^{1,0})$, where $p_k$ are invariant polynomials. By Theorem \ref{Theorem_convergencespectralcover}, after passing to a subsequence, $\kappa(\hat{\vp}_i)$ converges smoothly to $\mbs_{\infty} = (p_{1,\infty}, \cdots, p_{r,\infty})$ in $\MB_X$. The spectral cover defined by $\mbs_{\infty}$ is
		\[
		S_{\infty} := \{\lambda \in \Omega_X^{1,0} \mid \lambda^r + p_{1,\infty} \lambda^{r-1} + \cdots + p_{r,\infty} = 0\}.
		\]
		We denote by $\Zta$ the discriminant locus of $S_{\infty}$.
		
		\textbf{Step 2: Convergence of the curvature.}
		
		By Proposition \ref{Thm_citeconvergence}, for $x\in X \setminus \Zta$, there exist $i_0$, $R_{\infty}$, and $d_{x,\infty}$ such that for all $i \geq i_0$, over $B_{R_{\infty}}(x)$ we have:
		\begin{equation*}
			\begin{split}
				&|\sqrt{-1}\Lambda F_{A_i}^{\perp}| \leq Ce^{-C' r_i d_{x,\infty}},\quad |[\hat{\vp}_i, \hat{\vp}_i^{\dagger}]| \leq Ce^{-C' r_i d_{x,\infty}}, \\
				&\|d_{A_i}^{\ast} F_{A_i}\|_{L^p} \leq Ce^{-C' r_i d_{x,\infty}} \|\nabla_{A_i} \hat{\phi}_i\|_{L^p}.
			\end{split}
		\end{equation*}
		
		Additionally, over $B_{R_{\infty}}(x)$, by Lemma \ref{Lemma_localL2controlofthecurvature}, we have $\int_{B_{R_{\infty}}} |F_{A_i}|^2 \leq C$. Define the Uhlenbeck singularity set as
		\[
		Z_{\Uh} := \bigcap_{0 < R} \left\{x \in X \setminus \Zta \mid \liminf_{i \to \infty} r^{4 - 2n} \int_{B_R(x)} |F_{A_i}|^2 dV \geq \epsilon_0 \right\}.
		\]
		Then, by Theorem \ref{Uhlenbeckcompactness}, over $X \setminus (Z_{\Uh} \cup \Zta)$, there exists a connection $A_{\infty}$ such that, after passing to a subsequence, $A_i$ converges weakly in $L_{1,\loc}^p$ to $A_{\infty}$ up to a unitary gauge transformation. By Proposition \ref{prop_higher_derivative_curvature}, over $B_{R_{\infty}}(x)$ we have $|\nabla_{A_i}^k F_{A_i}| \leq C$. By standard bootstrapping estimates for connections, $A_i$ converges to $A_{\infty}$ in the $\MC^{\infty}_{\loc}$ sense.
		
		\textbf{Step 3: Convergence of the Higgs fields.}
		
		By Proposition \ref{Prop_L2controlLinfty}, for the rescaled Higgs field $\hat{\vp}_i$, we have $\|\hat{\vp}_i\|_{L^{\infty}(X)} \leq C$. Let $a_i = A_i - A_{\infty}$. Since $\bar{\pa}_{A_i} \hat{\vp}_i = 0$, we have $\bar{\pa}_{A_{\infty}} \hat{\vp}_i = a_i \hat{\vp}_i$. As $a_i$ converges weakly to $a_{\infty}$ in $C^{\infty}_{\loc}$, applying bootstrapping to $\hat{\vp}_i$, we conclude that there exists $\hat{\vp}_{\infty}$ such that $\hat{\vp}_i$ converges weakly in $\MC^{\infty}_{\loc}$ to $\hat{\vp}_{\infty}$. Define $\phi_{\infty} = \hat{\vp}_{\infty} + \hat{\vp}_{\infty}^{\dagger}$.
		
		\textbf{Step 4: The limiting pair $(A_{\infty}, \phi_{\infty})$ satisfies the adiabatic Hitchin-Simpson equations.}
		
		By Proposition \ref{Thm_citeconvergence} and Proposition \ref{Thm_curvatureconvergencewithderivativeestimate}, we have $\Lambda F_{A_{\infty}}^{\perp} = 0$. By Lemma \ref{Lemma_localL2controlofthecurvature}, we obtain $\phi_{\infty} \wedge \phi_{\infty} = 0$ and $d_{A_{\infty}} \phi_{\infty} = 0$, which is equivalent to the adiabatic Hitchin-Simpson equations \eqref{adabatic-Hitchin-Simpson}.
	\end{proof}
	
	In the case where $\rank(E) = 2$, we can remove the assumption that $Z_T \subsetneq X$ in Theorem \ref{Thm_mainconvergenceunboundedsequences}.
	
	\begin{proposition}
		Under the assumptions of Theorem \ref{Thm_mainconvergenceunboundedsequences}, suppose $\rank(E) = 2$. Then $\Zan$ can be chosen to have codimension at least 2.
	\end{proposition}
	
	\begin{proof}
		Let $(A_i, \phi_i = \vp_i + \vp_i^{\dagger})$ be a sequence of solutions with unbounded $r_i = \|\phi_i\|_{L^2(X)}$. We have $p_1(\hat{\vp}_i) = \Tr(\hat{\vp}_i) \to p_{1,\infty}$ and $p_2(\hat{\vp}_i) = \Tr(\hat{\vp}_i^2) \to p_{2,\infty}$.
		
		If $p_{1,\infty} = 0$, then by Proposition \ref{Prop_discriminantlocus}, the locus $\Zta$ has codimension at least 2. If $p_{1,\infty} \neq 0$, define $\beta_i = \vp_i - \frac{\Tr(\vp_i)}{2} \Id_E$ and consider the new sequence $(A_i, \beta_i + \beta_i^{\dagger})$, where $p_{1,\infty} = 0$. If $p_{2,\infty} \neq 0$, we set $\Zan := p_{2,\infty}^{-1}(0)$. The convergence of $r_i^{-1} \vp_i$ follows from the convergence of $\beta_i$ and $\Tr(\vp_i)$.
		
		When $p_{2,\infty} = 0$, we have the uniform bound $\|\beta_i\|_{L^{\infty}(X)} \leq C$. In this case, we can set $\Zan = \emptyset$ and apply Theorem \ref{Thm_Uhlenbeckbubbleboundedsequence}. As $\lim r_i \to \infty$, we obtain $\lim_{i \to \infty} r_i^{-1} \|\beta_i\| = 0$. Therefore, the limiting Higgs field in this case is $\vp_{\infty} = \Tr(\vp_{\infty}) \frac{\Id}{2}$.
	\end{proof}
	
	\begin{remark}
		When $\dim(X) = 4$, over a closed K\"ahler manifold, it was shown in \cite{tanaka2019singular} that the Hitchin-Simpson equations coincide with the Kapustin-Witten equations. For rank two traceless Higgs bundles, our theorem will be covered by Taubes' compactness theorem \cite{taubes2013compactness} for the Kapustin-Witten equations. Additionally, \cite{tanaka2019singular} shows that for rank two traceless Higgs bundles, the set $\Zta$ defined in Theorem \ref{Thm_mainconvergenceunboundedsequences} is nearly the same as the singular set found by Taubes in \cite{taubes2013compactness}, which justifies our use of the label "$T$" for the singular set $\Zta$.
	\end{remark}

\subsubsection{Uniqueness conjecture for canonical non-compact sequencs.}
Given a stable, generically semi-simple Higgs bundle $(\bar{\pa}_E, \vp)$, let $t \in \mathbb{R}$ be a real parameter. Then $(\bar{\pa}_E, t\vp)$ forms a non-compact family of stable Higgs bundles. If we denote by $(A_t, \phi_t)$ the corresponding family of solutions to the Hitchin-Simpson equations, this gives rise to a natural non-compact family with $r_t := \|\phi_t\|_{L^2} \to \infty$ when $t\to \infty$. By Theorem \ref{Thm_mainconvergenceunboundedsequences}, there exists a singular set $Z$ and a subsequence $t_i \to \infty$ such that, over $X \setminus Z$, the limit $\lim_{i \to \infty}(A_{t_i}, r_{t_i}^{-1} \phi_{t_i})$ converges to a solution satisfying the adiabatic Hitchin-Simpson equations \eqref{adabatic-Hitchin-Simpson}. However, it is not known whether this limit is independent of the choice of subsequences $t_i$.

The uniqueness of this limit plays an important role in understanding the asymptotic geometry of the moduli space of Higgs bundles, as discussed in \cite{mazzeo2019asymptotic, fredrickson2019exponential}. When $X$ is a Riemann surface, let $S_{\vp}$ denote the spectral variety of $\vp$. It was shown in \cite{mazzeo2012limiting, mochizuki2023asymptotic, fredrickson2018generic} that if the spectral variety $S_{\vp}$ is smooth, the limit as $t \to \infty$ is unique. However, when $S_{\vp}$ is not smooth, recent work by Na \cite{na2022limiting} shows that the limit may depend on the choice of subsequences.

We propose the following conjecture regarding the uniqueness of the limiting configuration over general K\"ahler manifolds, which will be of interest in further studies of the asymptotic geometry of the moduli space of Higgs bundles.

\begin{conjecture}
	Let $S_{\vp}$ be the spectral variety of $\vp$. If $S_{\vp}$ is smooth, then the limit of $(A_t, \phi_t)$ as $t \to \infty$ is independent of the choice of subsequences.
\end{conjecture}

	\end{section}

\section{Applications}
In this section, we discuss several applications of the compactness results for the Hitchin-Simpson equations. First, we introduce the generalized Hitchin WKB problem. Then, we explore the deformation problem of $\ZT$-harmonic 1-forms.
\begin{subsection}{Generalized Hitchin's WKB Problem}
	Let $X$ be a K\"ahler manifold (not necessarily compact), and let $(E, \bar{\pa}, t\vp)$ be a family of stable Higgs bundles, with $\deg(E) = 0$ and $\ch_2(E)[\omega]^{n-2} = 0$. Denote by $(A_t, \phi_t)$ the corresponding solutions to the Hitchin-Simpson equations, and let $D_t = d_{A_t} + \phi_t$ be the associated flat connection.
	
	When $X$ is a Riemann surface, Hitchin's WKB problem was originally introduced by Katzarkov-Noll-Pandit-Simpson in \cite{katzarkov2015harmonic}, and it was solved by Mochizuki in \cite{Mochizukiasymptotic} for the Riemann surface case. For the family of connections $D_t$, Hitchin's WKB problem seeks to understand the asymptotic behavior of the monodromy metric defined by $D_t$.
	
	Specifically, let $\gamma: [0, 1] \to X$ be a smooth embedded path. Denote by $P_t$ the parallel transport operator between $\gamma(0)$ and $\gamma(1)$. Let $H_{0,t}$ and $H_{1,t}$ be the harmonic metrics at $\gamma(0)$ and $\gamma(1)$, respectively. Denote by $P_t^{\ast} H_{1,t}$ the pullback metric under the parallel transport. The aim is to understand the following question. Since the original question for the Riemann surface case is called Hitchin's WKB problem, we refer to the following asymptotic problem as the generalized Hitchin's WKB problem.
	
	\begin{question}[Generalized Hitchin's WKB problem]
		\label{question_WKB}
		What is the asymptotic behavior of $P_t$ as $t \to \infty$? More precisely, what is the asymptotic behavior of the distance between $H_{0,t}$ and $P_t^{\ast} H_{1,t}$?
	\end{question}
	
	To address this, we first introduce some background on the distance between metrics. Let $V$ be a rank-$r$ vector space, and let $H_0, H_1$ be two Hermitian metrics on $V$. We choose an orthonormal basis $e_1, \dots, e_r$ for $V$ with respect to both metrics. Define $\alpha_i := \log|e_i|_{H_0} - \log|e_i|_{H_1}$, and the distance between the two Hermitian metrics is given by $\vec{d}(H_0, H_1) := (\alpha_1, \dots, \alpha_r)$. After reordering, we assume that $\alpha_1 \geq \alpha_2 \geq \dots \geq \alpha_r$.
	
	Now, let $(V_0, H_0)$ and $(V_1, H_1)$ be two Hermitian vector spaces, and let $P: V_0 \to V_1$ be an isomorphism. The operator norm of $P$ is defined as $\|P\|_{\Op} := \sup_{v_0 \in V_0} \left\{ \frac{|Pv_0|_{H_1}}{|v_0|_{H_0}} \right\}$. The following lemma compares the distance between $H_0$ and $P^{\ast} H_1$:
	
	\begin{lemma}[{\cite[Section 2.3.1]{Mochizukiasymptotic}}]
		\label{Lem_affinematrix}
		We can choose orthonormal bases $(e_1, \dots, e_r)$ for $V_0$ and $(f_1, \dots, f_r)$ for $V_1$ such that:
		\begin{enumerate}
			\item $P(e_i) = e^{\alpha_i} f_i$, where the $\alpha_i$ are real numbers ordered as $\alpha_1 \geq \alpha_2 \geq \dots \geq \alpha_r$,
			\item $\vec{d}(H_0, P^{\ast} H_1) = (\alpha_1, \dots, \alpha_r)$,
			\item Let $\Lambda^k P: \Lambda^k V_1 \to \Lambda^k V_2$ be the induced morphism. Then for any $k$, we have:
			\[
			\alpha_1 = \log \|P\|_{\Op}, \quad \alpha_k = \log \|\Lambda^k P\|_{\Op} - \log \|\Lambda^{k-1} P\|_{\Op}.
			\]
		\end{enumerate}
	\end{lemma} 
\end{subsection}
\begin{subsection}{The Asymptotic Behavior of the Parallel Transport}
	Before introducing the main theorem, we present a singular perturbation result developed by Mochizuki \cite{Mochizukiasymptotic}. Let $E$ be a vector bundle with a frame $\mbe$ over $[0,1]$, and let $A_0, A_1$ be diagonal matrices with smooth functions defined over $[0,1]$, while $B$ is an off-diagonal matrix over $[0,1]$. Let $\|\cdot\|_{\MC^0}$ denote the $\MC^0$ norm of a matrix. We then have the following proposition:
	
	\begin{proposition}[{\cite[Prop 2.18, Cor 2.19]{Mochizukiasymptotic}}]
		\label{singularperturbation}
		For $t>0$, let $D_t$ be a family of connections over $E$. Suppose that, in the frame $\mbe$, we can express $D_t \mbe = \mbe(tA_0 + A_1 + B)ds$. Then there exists a constant $\epsilon_0$ depending on $\|A_0\|_{\MC^0}$ and $\|A_1\|_{\MC^0}$, but independent of $t$, such that if $\|B\|_{\MC^0} \leq \epsilon_0$, the following hold:
		\begin{itemize}
			\item [(i)] There exist a diagonal matrix $K_t$ and an off-diagonal matrix $G_t$ such that under the frame $\mathbf{f} := \mbe(I + G_t)^{-1}$, we have $D_t\mathbf{f} = \mathbf{f}(tA_0 + A_1 + K_t)$.
			\item [(ii)] The matrices $K_t$ and $G_t$ satisfy the estimate:
			\begin{equation*}
				\|G_t\|_{\MC^0} + \|K_t\|_{\MC^0} + \|\partial_s G_t+[tA_0+A_1,G_t]\|_{\MC^0} \leq C\|B\|_{\MC^0}.
			\end{equation*}
		\end{itemize}
	\end{proposition}
	
	\begin{definition}
		A smooth embedded path $\gamma: [0,1] \to X$ is called non-critical if:
		\begin{itemize}
			\item [(i)] There exists a tubular neighborhood $U$ of $\gamma$ such that $(E, \bar{\pa}_E, \vp)|_{U} = \oplus_{i=1}^r (E_i, \bar{\pa}_i, \vp_i)|_U$ with $\rank(E_i) = 1$.
			\item [(ii)] For any point $s \in [0,1]$, we have $\gamma^{\ast} \Re(\vp_i)|_s \neq \gamma^{\ast} \Re(\vp_j)|_s$ for any $i \neq j$.
		\end{itemize}
	\end{definition}
	
	Let $\gamma$ be a non-critical path. Since $\gamma^{\ast}\Re(\vp_i)$ is a real 1-form on $[0,1]$, we define the number:
	\[
	\alpha_i := -\int_0^1 \gamma^{\ast} \Re(\vp_i) ds.
	\]
	Denote by $P_t: E_{\gamma(0)} \to E_{\gamma(1)}$ the parallel transport defined by $D_t$ along $\gamma$. Let $H_0 := H|_{E_{\gamma(0)}}$ and $H_1 := H|_{E_{\gamma(1)}}$ be the induced metrics on $E_{\gamma(0)}$ and $E_{\gamma(1)}$, respectively. Denote by $P^{\ast} H_1$ the pullback metric on $E_{\gamma(0)}$ under the parallel transport $P_t$. We then have the following generalized asymptotic theorem, which answers the generalized Hitchin's WKB problem stated in Question \eqref{question_WKB}.
	
	\begin{theorem}
		There exist positive constants $C$, $C'$, and $t_0$ depending on $(\bar{\pa}_E, \vp)$ such that for any non-critical path $\gamma$, there exists a constant $d$ depends on $\gamma$ such that
		\begin{equation}
			\left|\frac{1}{t} \vec{d}(H_0, P_t^{\ast} H_1) - (2\alpha_1, \dots, 2\alpha_r)\right| \leq Ce^{-C't d}
		\end{equation}
		for all $t \geq t_0$.
	\end{theorem}
	
	\begin{proof}
		For any non-critical path $\gamma$, since it does not pass through any discriminant locus, we can find:
		\begin{itemize}
			\item [(i)] A finite number of points $x_1, \dots, x_k$ with $x_i = \gamma(s_i)$ and $0 = s_1 \leq s_2 \leq \cdots \leq s_k = 1$.
			\item [(ii)] Open neighborhoods $U_i$ and constants $d_{x_i}$ such that, over $U_i$, the estimates from Section \ref{section_localestimate} hold for the Higgs bundle $(\bar{\pa}_E, \vp)$. We choose $d=\min_i d_{x_i}$.
		\end{itemize}
		
		Let $U = \cup_{i=1}^k U_i$. Then over $U$, by the non-criticality assumption of the path, we have the decomposition $(E, \bar{\pa}_E, t\vp) = \oplus_{i=1}^k (E_i, \bar{\pa}_i, t\vp_i)$. For the harmonic metric $H_t$, we define the approximate metric $H_t^0$ (defined in \eqref{eq_approximate_metric}), which is diagonal with respect to the decomposition, and $H_t^0|_{E_i} = H_t|_{E_i}$. By Proposition \ref{prop_approximate_Chern-estimate}, for the approximate Chern connection $D_t^0=d_{A_t^0}+t\phi_t^0$ defined in \eqref{eq_approximate_Chern}) be associated with the diagonal metric $H_t^0$, we have:
		\[
		|D_t - D_t^0|_{H_t} \leq Ce^{-C' t d}.
		\]
		
		Consider the restriction of the vector bundles, metrics, and connections over $\gamma(s)$, and write $\gamma^{\ast} (E,\bar{\pa}_E) = \oplus_{i=1}^r (\gamma^{\ast} E_i,\bar{\pa}_i|_{\gamma}$) with the induced metric $\gamma^{\ast} H$ and connections $\gamma^{\ast} D_t$ and $\gamma^{\ast} D_t^0$. We choose a frame $\mathbf{e} = (e_1, \dots, e_r)$ at $\gamma^{\ast} (E|_{\gamma(0)})$. As the unitary connection $d_{A_t^0}$ preserves the decomposition, we extend the frame $\mathbf{e}$ over $\gamma$ using the parallel transport defined by $d_{A_t^0}$. The frame $\mbe$ is normal with respect to $H_t^0$ and $H_t$, and orthogonal with respect to $H_t^0$, but it is not orthogonal with respect to $H_t$. However, by Corollary \ref{corocommutevanishofpi}, we have $|H_t(e_i, e_j)| \leq Ce^{-C't d}$ for $i\neq j$.
		
		By the definition of $D_t^0$ and the frame $\mfe$, there exists a diagonal matrix $M$ such that, in the frame $\mathbf{e}$, we have:
		\[
		D_t^0 \mathbf{e} = t \mathbf{e} M ds,
		\]
		where
		\[
		M ds = \gamma^{\ast} (\vp + \bar{\vp}) = \begin{pmatrix}
			2\gamma^{\ast} \Re (\vp_1) & 0 & \cdots & 0 \\
			0 & 2\gamma^{\ast} \Re (\vp_2) & \cdots & \cdots \\
			\vdots & \vdots & \vdots & \vdots \\
			0 & \cdots & 0 & 2\gamma^{\ast} \Re (\vp_r)
		\end{pmatrix} .
		\]
		
		Furthermore, we can write:
		\[
		\gamma^{\ast} D_t \mbe = \mbe (tM + B_1 + B_2) ds,
		\]
		where $B_1$ is diagonal (i.e., $(B_1)_{ij} = 0$ for $i \neq j$) and $B_2$ is off-diagonal (i.e., $(B_2)_{ii} = 0$). By Proposition \ref{prop_approximate_Chern-estimate}, we have $|B_1| + |B_2| \leq Ce^{-C't d}$.
		
		For large $t$, by the singular perturbation result (Theorem \ref{singularperturbation}), there exist diagonal matrix $K_t$ and off-diagonal matrix $G_t$ defined on $\gamma$ such that: under the frame $\mathbf{f} := \mbe(I + G_t)^{-1}$, we could write $$\gamma^{\ast} D_t \mathbf{f} = \mathbf{f}(tM + B_1 + K_t) ds$$ with $|G_t|+|K_t|_{\MC^0}\leq Ce^{-C't d}$.
		
		Thus, the parallel transport matrix $P_t$ of $\gamma^{\ast} D_t$ from $E_{\gamma(0)}$ to $E_{\gamma(1)}$ under the frames $\mbe(0)$ and $\mbe(1)$ can be written as:
		\[
		P_t := (\Id + G_t(1)) e^{-\int_0^1 (tM + B_1 + K_t) ds} (\Id + G_t(0))^{-1}.
		\]
		
		Let $\mathbf{g}(s) = (\mathbf{g}_1(s), \dots, \mathbf{g}_r(s))$ be the orthonormal frames of $E|_{\gamma(s)}$ constructed using the Gram-Schmidt process for the frame $\mbe$. Then we can write $\mathbf{g}(s) = \mbe(s)(\Id + R_t(s))$ with $|R_t(s)| \leq Ce^{-C't d}$. The parallel transport matrix $P_t$ under the frames $\mathbf{g}(0)$ and $\mathbf{g}(1)$ can then be expressed as:
		\[
		Z_{\gamma} := (\Id + L_t(1)) e^{-\int_0^1 (tM + B_1 + K_t) ds} (\Id + L_t(0))^{-1},
		\]
		where $L_t(s):=(\Id+R_t(s))(\Id+G_t(s))-\Id$.
		
		Since $tM + B_1 + K_t$ is diagonal with $|B_1|+|K_t|\leq Ce^{-C'td}$, we have:
		\[
		\left|\int_0^1 (tM + B_1 + K_t) ds - 2t \diag(\alpha_1, \dots, \alpha_r)\right| \leq Ce^{-C't d}.
		\]
		
		Recall that for a rank-$r$ matrix $Z$, $\|Z\|_{\Op} := \sup \{|Zv| \mid v \in \mathbb{C}^r, |v| = 1\}$. We have:
		\begin{equation}
			\log \|Z_{\gamma}\|_{\Op} \leq \log \|e^{-\int_0^1 (tM + B_1 + K_t) ds}\|_{\Op} + \mathcal{O}(e^{-C't d}),
		\end{equation}
		where $\mathcal{O}(e^{-C't d})$ denotes an error term controlled by $Ce^{-C't d}$.
		
		Since $\alpha_1 \geq \alpha_2 \geq \cdots \geq \alpha_r$, we have:
		\begin{equation}
			\log \|e^{-\int_0^1 (tM + B_1 + K_t) ds}\|_{\Op} \leq 2t\alpha_1 + \mathcal{O}(e^{-C't d}).
		\end{equation}
		
		On the other hand, we have:
		\begin{equation}
			|P_t \mathbf{g}_1(0)|_{H_{\gamma(1)}} = e^{2t\alpha_1}(1 + \mathcal{O}(e^{-C't d})).
		\end{equation}
		
		Thus, we conclude:
		\begin{equation}
			|\log \|P_t\|_{\Op} - 2t\alpha_1| \leq Ce^{-C't d}.
		\end{equation}
		
		Applying the same argument to $\Lambda^k P_t$, we obtain:
		\begin{equation}
			|\log \|\Lambda^k P_t\|_{\Op} - 2t \sum_{i=1}^k \alpha_i| \leq Ce^{-C't d}.
		\end{equation}
		
		By Lemma \ref{Lem_affinematrix}, the claim follows.
	\end{proof}
\end{subsection}
\subsection{Deformation Problem of the $\ZT$ Harmonic 1-Form}
In this section, we study the rank two version of the Hitchin-Simpson equations and its relationship with the previous work of Taubes \cite{Taubes20133manifoldcompactness,taubes2013compactness}. Moreover, we will discuss the deformation problem of $\ZT$ harmonic 1-forms.

Let $\MMH^2$ be the moduli space of rank two traceless Higgs bundles (see \eqref{eq_moduli_ranktwo_traceless}), and let $\MB_X^2$ be the rank two spectral base (see \eqref{eq_spectral_base_Ranktwo}). We now focus on the compactness theorem for sequences of solutions of rank two traceless Higgs bundles.

We begin by considering a solution to the adiabatic Hitchin-Simpson equations \eqref{adabatic-Hitchin-Simpson}. The following discussion is based on Taubes \cite[Section 7a]{Taubes20133manifoldcompactness}, which we briefly review here.

Let $Z \subset X$ be a Hausdorff codimension-four set, and let $(E,H)$ be a rank two Hermitian vector bundle on $X \setminus Z$. Suppose $(A, \phi = \vp + \vp^{\dagger})$ is a solution to \eqref{adabatic-Hitchin-Simpson}. The equation satisfied by $(A, \phi)$ can be rewritten as:
\begin{equation}
	\Lambda F_A^{\perp} = 0, \quad F_A^{0,2} = 0, \quad \phi \wedge \phi = 0, \quad d_A \phi = 0, \quad d_A^* \phi = 0, \quad \|\phi\|_{L^2} = 1.
\end{equation}

Let $\mfg_E$ be the adjoint bundle of $(E,H)$. Take $x\in X \setminus Z$ with a small open neighborhood $U_p \subset X \setminus Z$. Choose $\sigma \in \mfg_E$ such that $|\sigma|_H = 1$. After applying a unitary gauge, the condition $\phi \wedge \phi = 0$ implies that $\phi = v \otimes \sigma$. The equations $d_A \phi = d_A^* \phi = 0$ further imply:
\begin{equation}
	\label{eq_adiabatic_intermedia_step}
	dv \otimes \sigma + v \otimes d_A \sigma = 0, \quad d^* v \otimes \sigma + *v \otimes d_A \sigma = 0.
\end{equation}
Since $|\sigma|_H = 1$ and $d_A$ is a unitary connection, we have $\langle \sigma, d_A \sigma \rangle = 0$. Taking the inner product of \eqref{eq_adiabatic_intermedia_step} with $\sigma$ and using $|\sigma|_H = 1$, we conclude:
\[
dv = d^* v = 0.
\]
However, there is ambiguity in the sign. Since $\sigma$ is globally defined on $X \setminus Z$ only up to a $\pm$ sign, $v$ is also defined globally on $X \setminus Z$ up to a $\pm$ sign. Therefore, we can write $\phi = (\pm v) \otimes (\pm \sigma)$.

Next, consider the Hitchin morphism for $(A, \phi)$. For $\vp = \phi^{1,0} = v^{1,0} \otimes \sigma$, we have $\kappa(\vp) = \Tr(\vp^2) = 2v^{1,0} \otimes v^{1,0}$, and $s_v := \bar{\pa}(v^{1,0} \otimes v^{1,0}) = 0$ over $X \setminus Z$. Since $\|v\|_{L^2} = \|\phi\|_{L^2} = 1$, $s_v$ extends to a holomorphic section on $X$, i.e., $s_v \in H^0(\Sym^2 \Omega_X^1)$. Thus, we obtain $\bv := \pm v$, which is a $\ZT$ harmonic 1-form.

The compactness theorem for rank two Higgs bundles could be stated as the follows:
\begin{theorem}
	\label{Thm_compactnessrank2}
	Let $(E,H)$ be a rank 2 Hermitian vector bundle over $X$, and let $(A_i, \phi_i = \vp_i + \vp_i^{\dagger})$ be a sequence of solutions to the Hitchin-Simpson equations with $\Tr(\phi_i) = 0$ and $r_i := \|\phi_i\|_{L^2(X)}$ unbounded. Then, passing to a subsequence, we have:
	\begin{enumerate}[label=(\roman*)]
		\item There exists $s_{\infty} \in \MB_X^2$ such that $\det(r_i^{-1} \vp_i)$ converges to $s_{\infty}$ in the $\MC^{\infty}(X)$ topology.
		\item The following holds:
		\begin{enumerate}
			\item $Z_{\Uh}$ is a Hausdorff codimension-four set, and $\Zta$ is a codimension-two complex analytic subset.
			\item There exists a Hermitian vector bundle $(E_{\infty}, H_{\infty})$ and a Hermitian-Yang-Mills connection $A_{\infty}$ with $\phi_{\infty}$.
			\item There exists a sequence of isomorphisms $g_i: \Aut(E_{\infty}, H_{\infty}) \to \Aut(E,H)$ such that the sequence $g_i^{\ast}(A_i, \phi_i)$ converges in the $\MC^{\infty}_{\loc}$ topology to $(A_{\infty}, \phi_{\infty})$ over $X \setminus (Z_{\Ta} \cup Z_{\Uh})$. Moreover, $\kappa(\phi_{\infty}^{1,0}) = s_{\infty}$.
		\end{enumerate}
	\end{enumerate}
\end{theorem}

\begin{proof}
	All statements follow directly from Theorem \ref{Thm_mainconvergenceunboundedsequences} and the previous discussions.
\end{proof}

Let $\bv$ be a $\ZT$ harmonic 1-form over $X$ with $s_{\bv} = \bv^{1,0} \otimes \bv^{1,0}\in \MB_X^2$ being the corresponding element in the spectral base. Let $(A_i, \phi_i)$ be a sequence of solutions to the Hitchin-Simpson equations. We say $(A_i, \phi_i)$ converges to $\bv$ if $\lim_{i \to \infty}(A_i, \phi_i)$ exists, and under the Hitchin map, $\frac{1}{\|\phi_i\|_{L^2}^2} \lim_{i \to \infty} \kappa(\phi_i) = C s_{\bv}$ for some constant $C$. A $\ZT$ harmonic 1-form $\bv$ is said to be deformable if there exists a sequence $(A_i, \phi_i)$ such that the limit exists and converges to $\bv$ in the above sense.

\begin{question}[Deformation Problem]
	Which $\ZT$ harmonic 1-forms over $X$ can be deformed into a sequence of solutions to the Hitchin-Simpson equations?
\end{question}

With the help of the Hitchin morphism, using Theorem \ref{Thm_compactnessrank2} and the construction of the Hitchin section (Theorem \ref{thm_heliu_surjective_Hitchin_morphism}), we can fully answer the deformation problem of $\ZT$ harmonic 1-forms over a K\"ahler manifold.

\begin{theorem}
	\label{thm_deformation_ZT_harmonic_form}
	Every $\ZT$ harmonic 1-form over a K\"ahler manifold can be deformed into a sequence of solutions to the Hitchin-Simpson equations that converge to $\bv$.
\end{theorem}

\begin{proof}
	Given any $\ZT$ harmonic 1-form $\bv = \pm v$, let $s_{\bv} = 2v^{1,0} \otimes v^{1,0} \in \MB_X^2$ be the corresponding rank-one symmetric differential. By Theorem \ref{thm_heliu_surjective_Hitchin_morphism}, there exists a polystable Higgs bundle $(\bar{\pa}_E, \vp)$ such that $\kappa(\vp) = \Tr(\vp^2) = s_{\bv}$. Without loss of generality, assume $\|\vp\|_{L^2} = 1$. For a real parameter $t$, $(\bar{\pa}_E, t \vp)$ is a family of polystable Higgs bundles. By Theorem \ref{thm_nonabelian_Hogde}, there exists a family $(A_t, \phi_t)$ of solutions to the Hitchin-Simpson equations \eqref{Hitchin-Simpson} such that $[(d_{A_t}^{0,1}, \phi_t^{1,0})] = [(\bar{\pa}_E, t \vp)]\in \MMH$. Let $r_t := \|\phi_t\|_{L^2}$. As $\lim_{t \to \infty} \|\kappa(\vp)\|_{L^2} = \lim_{t \to \infty} t^2 \|s_{\bv}\|^2 = \infty$, $r_t$ is unbounded. By Theorem \ref{Thm_compactnessrank2}, there exists a subsequence $t_i$ such that $\lim_{i \to \infty}(A_{t_i}, \phi_{t_i})$ exists. By definition, we have $\kappa(\phi_t^{1,0}) = \kappa(t \vp) = t^2 s_{\bv}$. By Corollary \ref{Cor_RelationshipL2pk}, for $\|\phi_t\|_{L^2}$, we have $C^{-1} t \leq \|\phi_t\|_{L^2} \leq C(t+1)$, which implies:
	\[
	\lim_{t_{i} \to \infty} \frac{1}{\|\phi_{t_i}\|_{L^2(X)}^2} \kappa(\phi_{t_i}^{1,0}) = C s_{\bv},
	\]
	completing the proof.
\end{proof}

Let $X$ be a closed 4-manifold, and let $(E,H)$ be a Hermitian vector bundle. Since $** = 1$, the Hodge star operator $*: \Omega^2 \to \Omega^2$ defines a decomposition $\Omega^2 = \Omega^{2,+} \oplus \Omega^{2,-}$, where $\Omega^{2,+}$ are the self-dual two-forms and $\Omega^{2,-}$ are the anti-self-dual two-forms. 

The Kapustin-Witten equations, introduced in \cite{KapustinWitten2006}, for a pair $(A, \phi)$ can be written as:
\begin{equation}
	\label{eq_KWequation}
	\begin{split}
		(F_A^{\perp} - \phi \wedge \phi)^+ = 0, \quad (d_A \phi)^- = 0, \quad d_A * \phi = 0,
	\end{split}
\end{equation}
which can be regarded as a generalization of the Yang-Mills equations. It was proved by Y. Tanaka \cite{tanaka2019singular} that over a K\"ahler surface, the Kapustin-Witten equations \eqref{eq_KWequation} are equivalent to the Hitchin-Simpson equations \eqref{Hitchin-Simpson}.

In Example \ref{example_bogomolov}, Bogomolov-Oliveira's example \cite{bogomolov2011symmetric} (denoted $X_{BO}$) is a complex K\"ahler surface that is simply connected but has non-vanishing $\ZT$ harmonic 1-forms. Since $X_{BO}$ is simply connected, there are no topologically trivial solutions to the Hitchin-Simpson equations. Therefore, by Theorem \ref{thm_deformation_ZT_harmonic_form}, we conclude the following:

\begin{corollary}
	Over the 4-manifold $X_{\mathrm{BO}}$, there exists a $\ZT$ harmonic 1-form that is the limit of a sequence of solutions to the Kapustin-Witten equations but not the limit of a sequence of solutions to flat $\mathrm{SL}_2(\mathbb{C})$ connections.
\end{corollary}

	\bibliographystyle{plain}
	\bibliography{references}
\end{document}